\documentclass[onefignum,onetabnum]{siamart190516}

\usepackage{amsfonts}
\usepackage{amsmath}
\usepackage{mathtools}
\usepackage{graphicx}
\usepackage[caption=false]{subfig} 
\usepackage{epstopdf}
\usepackage[normalem]{ulem}
\usepackage{algorithmic}
\usepackage{tikz,pgfplots}
\pgfplotsset{compat=1.15}
\usetikzlibrary{arrows,shapes,calc,fit}
\usepackage{amsopn}
\usepackage{cleveref}
\usepackage{algorithmic}
\Crefname{ALC@unique}{Line}{Lines}
\ifpdf
\DeclareGraphicsExtensions{.eps,.pdf,.png,.jpg}
\else
\DeclareGraphicsExtensions{.eps}
\fi

\newsiamremark{remark}{Remark}
\newsiamremark{hypothesis}{Hypothesis}
\crefname{hypothesis}{Hypothesis}{Hypotheses}
\newsiamthm{claim}{Claim}
\newsiamthm{assumption}{Assumption}

\headers{A structure preserving shift-invert Infinite Arnoldi algorithm}{P. Appeltans, and W. Michiels}

\title{A structure preserving shift-invert infinite Arnoldi algorithm for a class of delay eigenvalue problems with Hamiltonian symmetry \thanks{Submitted to the editors on \today. \funding{This work was supported by the project C14/17/072 of the KU Leuven Research Council and the project G092721N of the Research Foundation-Flanders (FWO - Vlaanderen).}}}

\author{Pieter Appeltans\thanks{NUMA Section, Department of Computer Science, KU Leuven, Leuven (B-3001), Belgium
		(\email{pieter.appeltans@kuleuven.be},\email{wim.michiels@kuleuven.be}).}
	\and Wim Michiels\footnotemark[2]}

\DeclareMathOperator{\diag}{diag}
\DeclareMathOperator{\Span}{span}
\newcommand{\hinf}{\mathcal{H}_{\infty}}
\newcommand{\operator}{\mathcal{H}}
\newcommand{\mappedoperator}{\mathcal{R}_{\sigma}}
\newcommand{\adjointoperator}{\mathcal{G}}
\newcommand{\hinfnorm}{$\mathcal{H}_{\infty}$-norm}
\newcommand{\bilinear}[2]{\mathbf{B}\left(#1, #2 \right)}
\newcommand{\R}{\mathbb{R}}
\newcommand{\C}{\mathbb{C}}
\newcommand{\ie}{i.e.,~}
\newcommand{\eg}{e.g.,~}
\newcommand{\dd}{\mathop{}\!\mathrm{d}}
\newcommand{\negative}{\scalebox{0.7}{\(-\)}}
\newcommand{\minus}{\scalebox{0.7}{\(\,-\,\)}}
\ifpdf
\hypersetup{
  pdftitle={A structure preserving shift-invert infinite Arnoldi algorithm for a class of delay eigenvalue problems with Hamiltonian symmetry},
  pdfauthor={P. Appeltans, and W. Michiels}
}
\fi

\begin{document}

\maketitle

\begin{abstract}
  In this work we consider a class of delay eigenvalue problems that admit a spectrum similar to that of a Hamiltonian matrix, in the sense that the spectrum is symmetric with respect to both the real and imaginary axis. More precisely, we present a method to iteratively approximate the eigenvalues closest to a given purely real or imaginary shift, while preserving the symmetries of the spectrum. To this end, the presented method exploits the equivalence between the considered delay eigenvalue problem and the eigenvalue problem associated with a linear but infinite-dimensional operator. To compute the eigenvalues closest to the given shift, we apply a specifically chosen shift-invert transformation to this linear operator and compute the eigenvalues with the largest modulus of the new shifted and inverted operator using an (infinite) Arnoldi procedure. The advantage of the chosen shift-invert transformation is that the spectrum of the transformed operator has a ``real skew-Hamiltonian''-like structure. Furthermore, it is proven that the Krylov space constructed by applying this operator satisfies an orthogonality property in terms of a specifically chosen bilinear form. By taking this property into account during the orthogonalization process, it is ensured that, even in the presence of rounding errors, the obtained approximation for, \eg{} a simple, purely imaginary eigenvalue is simple and purely imaginary. The presented work can thus be seen as an extension of [V. Mehrmann and D. Watkins, \textit{Structure-Preserving Methods for Computing Eigenpairs of Large Sparse Skew-Hamiltonian/Hamiltonian Pencils}, \textsc{SIAM J. Sci. Comput.} (22.6), 2001], to the considered class of delay eigenvalue problems. Although the presented method is initially defined on function spaces, it can be implemented using finite-dimensional linear algebra operations. The performance of the resulting numerical algorithm is verified for two example problems: the first example illustrates the advantage of proposed approach in preserving purely imaginary eigenvalues when working in finite precision, while the second one demonstrates its applicability to a large scale problem.
\end{abstract}



	\section{Introduction}
In this manuscript, we consider non-linear eigenvalue problems (NLEVPs) of the form
\begin{equation}
\label{eq:nlevp}
M(\lambda) v = 0, 
\end{equation}
with $\lambda\in \C$ an eigenvalue and $v \in \C^{2n}\setminus \{0\}$ a right eigenvector, for which the characteristic matrix has the following structure:
\begin{equation}
\label{eq:characteristic_matrix}
M(\lambda):= \lambda I_{2n} - H_0 - \sum_{k=1}^{K} \left(H_{-k} e^{-\lambda \tau_k} + H_{k} e^{\lambda \tau_k}\right),
\end{equation}
with $0<\tau_1 < ... < \tau_K<\infty$ discrete delays and $I_{2n}$ the identity matrix of size $2n$. The matrices $H_0,H_1,H_{-1}, \dots, H_K$ and $H_{-K}$ belong to $\R^{2n \times 2n}$ and satisfy the following assumption.

\begin{assumption}
	\label{assumption:hamiltonian_structure}
	Firstly, the matrix $H_0$ is Hamiltonian meaning that 
	\begin{equation}
	\label{eq:hamiltonian_H0}
	\big(J H_{0}\big)^{\top} = JH_{0}
	\end{equation}
	in which the matrix $J$ is defined as
	\[
	J :=\begin{bmatrix}
	0 & I_n \\ -I_n & 0
	\end{bmatrix},
	\]
	with $I_{n}$ the identity matrix with size $n$.	Secondly, the matrices $H_k$ and $H_{-k}$ are related via
	\begin{equation}
	\label{eq:hamiltonian_Hk}
	\big(JH_{-k}\big)^{\top} = JH_{k} \text{ for }k=1,\dots,K,
	\end{equation} 
	with $J$ as defined above.
	
\end{assumption}

A motivation for considering NLEVPs with a characteristic matrix of the form \eqref{eq:characteristic_matrix} stems from a popular approach to compute the \hinfnorm{}, of a time-delay system \cite{Gumussoy2011a}. More specifically, consider the following state-space system with delays
\begin{equation}
\label{eq:dyn_sys}
\left\{
\begin{array}{lcl}
\dot{x}(t) &=& A_0 x(t) + \sum_{k=1}^{K} A_k x(t-\tau_k) + B w(t)\\
z(t) & =& C x(t),\\
\end{array}
\right.
\end{equation}
with $x(t) \in \R^{n}$ the state vector, $w(t)\in \R^{m}$ the performance input, $z(t)\in \R^{p}$ the performance output, $0<\tau_1 < ... < \tau_K<\infty$ discrete delays, and  $A_0$, \dots, $A_K$, $B$ and $C$ real-valued matrices of appropriate dimensions. The corresponding transfer matrix, 
\begin{equation*}
T(s) := C\left(s I - A_0 - \textstyle\sum\limits_{k=1}^{K} A_k e^{-s \tau_k}\right)^{-1} B,
\end{equation*} describes the system's input-output map in the frequency domain. If system \eqref{eq:dyn_sys} is exponentially stable, its \hinfnorm{}  is given by
\begin{equation*}
\|T(\cdot)\|_{\hinf} := \max_{\omega \in \R^{+}} \|T(\jmath \omega)\|_2,
\end{equation*} with $\jmath$ the imaginary unit. The \hinfnorm{} is an important performance measure in the robust control framework as it can be used to quantify both the input-to-output noise suppression of the system as well as the distance to instability of the system \cite{Zhou1998,Hinrichsen2005}.
To compute the \hinfnorm{} of a time-delay system, the methods presented in \cite{Michiels2010,Gumussoy2011a}, which can be seen as an extension of the well-known Boyd-Balakrishnan-Bruinsma-Steinbuch algorithm \cite{Boyd1990,Bruinsma1990a}, use a level set approach. An important component of such level set algorithms is to check whether for a given $\gamma>0$ the inequality $\|T(\cdot)\|_{\hinf}\geq \gamma$ holds. For systems of the form \eqref{eq:dyn_sys}  the following equivalence from \cite[Lemma~2.1]{Michiels2010} can be used: for $\omega \in \R$ the matrix $T(\jmath\omega)$ has a singular value equal to $\gamma$ if and only if  $\jmath\omega$ is an eigenvalue of the NLEVP associated with the following characteristic matrix,
\begin{equation}
\label{eq:nlevp_hinf}
\setlength{\arraycolsep}{3pt}
\lambda \begin{bmatrix}
I_n & 0 \\
0 & I_n
\end{bmatrix} - \begin{bmatrix}
A_0 & \gamma^{-1} BB^{\top} \\
-\gamma^{-1}C^{\top}C & -A_0^{\top}
\end{bmatrix} -\sum_{k=1}^{K} \left(\begin{bmatrix}
A_k & 0 \\
0 & 0
\end{bmatrix} e^{-\lambda \tau_k}+\begin{bmatrix}
0 & 0 \\
0 & -A_k^{\top}
\end{bmatrix} e^{\lambda \tau_k}\right).
\end{equation}
Notice that this characteristic matrix fits the structure given in \eqref{eq:characteristic_matrix}. Verifying whether $\|T(\cdot)\|_{\hinf}\geq \gamma$ is thus equivalent with checking whether the NLEVP associated with \eqref{eq:nlevp_hinf} has purely imaginary eigenvalues. For this application, it is thus important to accurately compute the purely imaginary eigenvalues. 

The spectrum, \ie the set of eigenvalues, of the considered NLEVP has some interesting features.
Firstly, notice that \eqref{eq:nlevp} bears some similarities with a retarded delay eigenvalue problem (RDEVP) \cite{Hale1993a,Gu2003a,Michiels2014}, but in contrast has both positive and negative delays. It is therefore not surprising that \eqref{eq:nlevp} generally has infinitely many eigenvalues. Secondly, due to the considered structure of the characteristic matrix and \Cref{assumption:hamiltonian_structure}, the considered class of NLEVPs can be seen as an extension of the linear eigenvalue problem 
\begin{equation}
\label{eq:evp}
\left(I_{2n}\lambda - H\right) v = 0,
\end{equation}
with $H\in \R^{2n \times 2n}$ a Hamiltonian matrix, as studied in, among others, \cite{Benner,VanLoan1984,Mehrmann2001}. It is well known that the spectrum of such eigenvalue problems is symmetric with respect to both the real and imaginary axis, meaning that eigenvalues appear either in quadruplets ($\lambda$, $\bar{\lambda}$, $-\lambda$ and $-\bar{\lambda}$) or in purely real or imaginary pairs ($\lambda$ and $-\lambda$) \cite{Mehrmann1991}. In the next section we will see that the spectrum of the considered NLEVP possesses the same symmetries.

The goal of this paper is to develop a method to accurately compute the eigenvalues of \eqref{eq:nlevp} that lie close to a given shift $\sigma$ while preserving the symmetry of the spectrum. We are particularly interested in the case for which the characteristic matrix \eqref{eq:characteristic_matrix} is large.	For large-scale linear eigenvalue problems, Krylov subspace methods, such as the (shift-invert) Arnoldi method, are well established \cite{Saad2003}. Recently, these methods have been generalized to NLEVPs, such as RDEVPs, see for example \cite{Jarlebring2010,Jarlebring2018,guttel2017,Voss2013,Jarlebring2012}. In the case of RDEVPs, the method presented in \cite{Jarlebring2010} can be interpreted in two ways. Firstly, it can be understood as applying the Arnoldi method to a sufficiently large linearisation of the original RDEVP with vectors of increasing length. Secondly, it can be interpreted as applying the Arnoldi method to an associated linear but infinite-dimensional operator, resulting in a Krylov subspace that is spanned by functions instead of vectors. This last interpretation gave rise to its name, the infinite Arnoldi method. The main benefit of this method is that it is in some sense discretization-free, as the error made by approximating the non-linear delay term can be reduced by applying more Arnoldi iterations. Here we want to apply a similar approach to compute the eigenvalues of \eqref{eq:nlevp} closest to a given purely real or imaginary shift. However, for eigenvalue problems with a Hamiltonian structure, as considered here, the regular shift-invert (infinite) Arnoldi method destroys the particular structure of the spectrum. For linear finite-dimensional eigenvalue problems, a modified shift-invert Arnoldi approach was therefore developed in \cite{Mehrmann2001}, which preserves the Hamiltonian structure. This approach was subsequently generalized to both matrix pencils and polynomial eigenvalue problems in \cite{Mehrmann2012} and \cite{Mehrmann2002}, respectively. The goal of this paper is to extend this method to the considered class of NLEVP.

The adopted approach can be interpreted as the algorithm of \cite{Mehrmann2001} applied to a linear operator whose spectrum corresponds to the one of \eqref{eq:nlevp}, instead of a matrix. An alternative approach could consist of constructing a rational or polynomial approximation of the characteristic matrix \eqref{eq:characteristic_matrix}, followed by a direct application of the algorithm of \cite{Mehrmann2002} (a so-called discretize-and-solve approach). As we are interested in computing selected eigenvalues of large-scale problems close to a shift $\sigma$ (on the real or imaginary axis) this would require an approximation that preserves the Hamiltonian symmetry of the spectrum and is accurate around $\pm \sigma$,  precluding an approximation around zero if $|\sigma|\gg0$. As we shall see, the latter property is naturally embedded in the presented shift-invert transformation of the aforementioned linear operator. Using the discretize-and-solve framework, it is also not trivial to determine a suitable degree for the discretization of the characteristic matrix such that the approximation error is sufficiently small (e.g., such that the imaginary axis eigenvalues are preserved for the \hinfnorm{} computation). It should be noted, though, that the discretize-and-solve framework does not exclude dynamic methods, for which the degree of the discretisation does not need to be specified beforehand and can change during the Arnoldi process (see, e.g., the dynamic variants of NLEIGS \cite{nleigs}). These methods are however closely related to the adopted infinite Arnoldi approach. \\

The remainder of this work is structured as follows. \Cref{sec:preliminary_results} recalls some preliminary results: first some properties of the spectrum of \eqref{eq:nlevp} are highlighted, next the most important components of the structure preserving shift-invert Arnoldi method from \cite{Mehrmann2001} are reviewed and finally the equivalence of the NLEVP associated with \eqref{eq:characteristic_matrix} and two linear infinite-dimensional eigenvalue problems, one related to the left eigenspace and one related to the right eigenspace, is demonstrated. In \Cref{sec:structure_preserving_delay}, the structure preserving shift-invert Arnoldi method from \cite{Mehrmann2001} is generalized to the considered NLEVP \eqref{eq:nlevp} using the aforementioned equivalence with a linear infinite-dimensional eigenvalue problem, giving rise to an algorithmic implementation of the structure preserving shift-invert \emph{infinite} Arnoldi method. Subsequently \Cref{sec:shift_zero,sec:non_zero_shift} discuss how this method can be implemented using finite-dimensional operations for the shift $\sigma$ employed in the shift-invert Arnoldi method equal to zero and different from zero, respectively. Next, \Cref{sec:numerical_illustration} applies the resulting numerical implementations on two example problems. The first example illustrates that the presented method indeed preserves the Hamiltonian structure of the spectrum, while the second example demonstrates its applicability to large-scale problems. Finally, \Cref{sec:conclusions} draws some concluding remarks.
\section{Preliminary results}
\label{sec:preliminary_results}
In this section we will recall some important preliminary results. 
\subsection{Properties of the considered NLEVP}
Throughout the paper a nonzero column vector $w\in \C^{2n}\setminus\{0\}$ is a left eigenvector of the characteristic matrix \eqref{eq:characteristic_matrix} associated with the eigenvalue $\lambda \in \C$ if it satisfies
\[
w^{\top} M(\lambda) = 0.
\]
Note that this definition differs from the most commonly used definition for the left eigenvector, namely, a nonzero column vector $w$ is a left eigenvector if there exists a $\lambda$ such that $w^{H}M(\lambda) = 0$. The notation used here is however common for real Hamiltonian eigenvalue problems and simplifies the notation in the remainder of this text.

Now we prove two important properties of the spectrum of the considered NLEVP \eqref{eq:nlevp}. Firstly, as mentioned before, the considered NLEVP typically has infinitely many eigenvalues. However the number of eigenvalues in any vertical strip around the imaginary axis is finite, as stated in the following proposition.

\begin{proposition}
	\label{prop:finite_number_eigenvalues_around_imag_axis}
	For any $c > 0$, the NLEVP \eqref{eq:nlevp} has only a finite number of eigenvalues in the vertical strip $\left\{z \in \C : -c < \Re(z) < c  \right\}$.
\end{proposition}
\begin{proof}
	This result follows from a similar argument as in \cite[Corollary~2.6]{Michiels2010}.
\end{proof}

Secondly, we show that the spectrum of \eqref{eq:nlevp} is symmetric with respect to both the real and imaginary axis. Symmetry with respect to the real axis follows from the fact that all matrices and delays in \eqref{eq:characteristic_matrix} are real-valued, while symmetry with respect to the origin then follows from the following proposition.
\begin{proposition}
	\label{proposition:hamiltonian_symmetry}
	If $\lambda$ is an eigenvalue of \eqref{eq:characteristic_matrix} with associated right eigenvector $v_{+}$ and left eigenvector $w_{+}$, then $-\lambda$ is also an eigenvalue of \eqref{eq:characteristic_matrix} and the corresponding right and left eigenvectors are $Jw_{+}$ and $Jv_{+}$, respectively.
\end{proposition}
\begin{proof}
		Since $(\lambda, w_{+})$ is a left eigenpair of $M(\lambda)$ and $J$ is nonsingular, we have that $w^{T}_{+}M(\lambda)J = 0$. Transposing this equation and using Assumption 1.1 leads to	$M(-\lambda)Jw_{+} = 0$, i.e., $(-\lambda, Jw_{+})$ is a right eigenpair of $M(\lambda)$. In a similar way it can be shown that if $(\lambda, v_{+})$ is a right eigenpair of $M(\lambda)$ then $(-\lambda, Jv_{+})$ is a left eigenpair of $M(\lambda)$.
\end{proof}

\subsection{The structure preserving shift-invert Arnoldi method for linear finite-dimensional Hamiltonian eigenvalue problems}
\label{subsec:finite_dimensional_case}
In this subsection we review the method presented in \cite{Mehrmann2001}, as our paper requires a good understanding of the results in the finite-dimensional case, which will be mirrored by their infinite-dimensional counterparts in \Cref{sec:structure_preserving_delay}. The method in \cite{Mehrmann2001} allows to compute the eigenvalues of \eqref{eq:evp} closest to a given purely real or imaginary shift $\sigma$ while preserving the Hamiltonian structure of the spectrum. Below we will assume $\sigma$ purely imaginary (the discussion for $\sigma$ purely real is similar). Let us first consider the traditional shift-invert Arnoldi method which applies the Arnoldi method to the shifted and inverted matrix
\[\left(H-\sigma I_{2n}\right)^{-1}.\]
Note that the transformed matrix is no longer real nor Hamiltonian and the purely imaginary eigenvalue $\jmath\omega$ of $H$ is mapped to the purely imaginary eigenvalue $\mu = \frac{1}{\jmath\omega-\sigma}$ of the matrix $\left(H-\sigma I\right)^{-1}$. However, when applying the standard Arnoldi procedure to the complex matrix  $\left(H-\sigma I\right)^{-1}$, the obtained approximation for $\mu$ typically does not lie on the imaginary axis since the Arnoldi procedure does not preserve the eigenvalue symmetry with respect to the imaginary axis in the constructed Hessenberg matrix. Furthermore, even if the Hessenberg matrix would contain the desired purely imaginary Ritz value, then computing the eigenvalues of this Hessenberg matrix in finite precision would introduce a small real component due to rounding errors inside the eigenvalue solver.  As a consequence, the computed approximations for purely imaginary eigenvalues of $H$ typically have a small but non-zero real part. This means that in applications for which the detection of purely imaginary eigenvalues is important, such as the application mentioned in the introduction, additional processing to determine whether an eigenvalue is purely imaginary, is necessary.

To avoid this additional processing, the structure preserving shift-invert Arnoldi method from \cite{Mehrmann2001} is preferred. For a purely real or imaginary shift $\sigma$, the Arnoldi method is now applied to the matrix 
\begin{equation}
\label{eq:shift_invert_matrix}
R_{\sigma}^{-1} :=\left(H+\sigma I_{2n} \right)^{-1}\left(H-\sigma I_{2n} \right)^{-1}.
\end{equation} 
The transformed matrix $R_{\sigma}^{-1}$ is real-valued and skew-Hamiltonian\footnote{A matrix $S$ is skew-Hamiltonian if $(JS)^{\top} = -JS$.}. Both the eigenvalues $\lambda$ and $-\lambda$ of $H$ are now mapped to the eigenvalue $\mu = \frac{1}{\lambda^2-\sigma^2}$ of $R_{\sigma}^{-1}$ (which consequently has multiplicity two). For the traditional Arnoldi method such multiple eigenvalues would hamper the convergence behavior. Yet, as we will see below, each eigenvalue of $R_{\sigma}^{-1}$ will only appear once in the Hessenberg matrix obtained by applying the Arnoldi procedure to $R_{\sigma}^{-1}$. More specifically, consider the following Krylov subspace generated by $R_{\sigma}^{-1}$,
\begin{equation}
\label{eq:krylov_space_fd}
\mathrm{K}_{m}\left(R_{\sigma}^{-1},q_1\right) = \Span\left\{q_1, R_{\sigma}^{-1}q_1, R_{\sigma}^{-2}q_1, \dots,R_{\sigma}^{-(m-1)}q_1\right\},
\end{equation}
with $q_1\in \R^{2n}$ an arbitrary real-valued starting vector. It can be shown that this subspace is $J$-neutral (sometimes also referred to as isotropic), meaning that for each pair of vectors $x$ and $y$ in this subspace, the equality $x^{\top}Jy = 0$ holds, see \cite[Proposition 3.3]{Mehrmann2001}. This has the following important consequence.
\begin{proposition}
	\label{proposition:dimension_intersection}
	Let $q_1$ be an arbitrary real-valued vector of length $2n$ and let $\lambda\neq 0$ be a simple eigenvalue of \eqref{eq:evp} with corresponding right eigenvector $v_{+}$ and let $v_{-}$ be a right eigenvector associated with the eigenvalue $-\lambda$ of $H$, then the dimension of the intersection of $\Span\{v_{+},v_{-}\}$ and $\mathrm{K}_m\left(R_{\sigma}^{-1},q_1\right)$ is at most 1.
\end{proposition}
\begin{proof}
	The result follows from a similar argument as in \cite[Lemma~2.3]{Mehrmann2012}, which is repeated here to ease the derivations for the infinite-dimensional case. It follows from \Cref{proposition:hamiltonian_symmetry} that if $(\lambda,v_{+})$ is a right eigenpair of $H$, then $(-\lambda,w_{-})$ with $w_{-} := Jv_{+}$ is a left eigenpair of $H$. This implies that $v_{-}^{\top}Jv_{+} = v_{-}^{\top}w_{-} \neq 0$, in which the last inequality follows from the fact that $-\lambda$ is a simple eigenvalue of $H$. Thus if the intersection of $\Span\{v_{+},v_{-}\}$ and $\mathrm{K}_m\left(R_{\sigma}^{-1},q_1\right)$ would have dimension 2, this would imply that both $v_{+}$ and $v_{-}$ are in $\mathrm{K}_m\left(R_{\sigma}^{-1},q_1\right)$, which would contradict the  $J$-neutrality of the constructed Krylov subspace.
\end{proof}
This proposition has an important effect on the obtained Ritz values as demonstrated below. For sake of simplicity, assume that all eigenvalues of $H$ are simple and different from zero. In this case, all eigenvalues of $R_{\sigma}^{-1}$ have multiplicity two and the right eigenspace corresponding to the eigenvalue $\mu$ is spanned by $v_{+}$ and $v_{-}$, the right eigenvectors of $H$ associated with $\lambda$ and $-\lambda$, respectively. Now let us introduce the matrix $Q_m = \begin{bmatrix}
q_1 & \dots & q_{m}
\end{bmatrix}$, with $m\leq n$, whose columns $q_1, \dots, q_{m} \in \R^{2n}$ form an orthonormal basis for $\mathrm{K}_m\left(R_{\sigma}^{-1},q_1\right)$. Using this notation, the Arnoldi recurrence relation can be written as
\begin{equation}
\label{eq:arnoldi_recurrence}
R_{\sigma}^{-1} Q_m = Q_m \Psi_m + \Psi_{[m+1,m]} q_{m+1} e_m^{\top} 
\end{equation}
with $\Psi_m \in \R^{m\times m}$ a reduced Hessenberg matrix, $\Psi_{[m+1,m]}$ a real-valued scalar, $q_{m+1} \in \R^{2n}$ and $e_m\in\R^{2n}$ the $m$\textsuperscript{th} Euclidean basis vector. When the Arnoldi procedure breaks down, \ie{} $\Psi_{[m+1,m]} = 0$, let $\Psi_m V_m = V_m \Sigma_m$ be an eigenvalue decomposition of $\Psi_m$ with $\Sigma_m = \diag(s_1,\dots,s_m)$ a diagonal matrix containing the eigenvalues of $\Psi_m$ (the so called Ritz values) and $V_m\in\C^{m\times m}$ a matrix containing the associated right eigenvectors. It then follows from \eqref{eq:arnoldi_recurrence} that  $R_{\sigma}^{-1} (Q_m V_m) = (Q_m V_m) \Sigma_m$, meaning that $s_1,\dots,s_m$ are also eigenvalues of $R_{\sigma}^{-1}$ with as corresponding right eigenvectors the columns of $Q_m V_m$. However, although each eigenvalue of $R_{\sigma}^{-1}$ has multiplicity two, there appear no doubles in $s_1,\dots,s_m$ as this would imply that there exist two linear independent vectors in $\Span\{v_{+},v_{-}\}$ that lie in the column space of $Q_m$ what would contradict \Cref{proposition:dimension_intersection}.

Thus, if the eigenvalues $\pm\jmath\omega$ lie sufficiently close to the shifts $\pm\sigma$, it follows from the reasoning above and the convergence behavior of the Arnoldi method that for sufficiently large $m$ the corresponding real eigenvalue $\mu = \frac{1}{-\omega^2-\sigma^2}$ of $R_{\sigma}^{-1}$ will be approximated by a \textbf{real and simple Ritz value} of the \textbf{real-valued} Hessenberg matrix $\Psi_m$. Thus even when the eigenvalues of $\Psi_m$ are computed in finite precision, the approximations for $\pm\jmath\omega$ obtained via the transformation $\pm\sqrt{\frac{1}{s}+\sigma^2}$, with $s$ the computed eigenvalues of $\Psi_m$, are typically strictly imaginary.

Note however that to ensure that only one component of the space $\Span\{v_{+},v_{-}\}$ appears in the Krylov subspace, it is important that the Krylov subspace remains $J$-neutral. Although this property is satisfied by construction when working in exact arithmetic, $J$-neutrality is typically lost when working in finite precision. \cite[Section~5]{Mehrmann2001} therefore suggests to let the orthogonalization procedure be based on the following Arnoldi recurrence relation, 
\begin{equation}
\label{eq:arnoldi_recurrence_2}
\Psi_{[m+1,m]} q_{m+1} = R_{\sigma}^{-1} q_m - Q_m \Psi_{[:,m]} - J Q_m \varUpsilon_{[:,m]}, 
\end{equation}
with $\Psi_{[:,m]} = Q_m^{\top} R_{\sigma}^{-1} q_m$, $\varUpsilon_{[:,m]}  = (J Q_m)^{\top} R_{\sigma}^{-1} q_m$ and $\Psi_{[m+1,m]}$ a normalisation factor such that $\|q_{m+1}\|_2 = 1$, \ie{} orthogonalise the new basis vector not only with respect to $Q_m$ but also to $JQ_m$. Recall that the last term in \eqref{eq:arnoldi_recurrence_2} is zero when working in exact arithmetic.

\begin{remark}
	Recall that $\sigma$ was assumed either purely real or purely imaginary. However, similar results for a more general shift $\sigma$ can be obtained by choosing $R_{\sigma}^{-1} = \Big((H-\sigma I_{2n})(H+\sigma I_{2n})(H-\bar{\sigma} I_{2n})(H+\bar{\sigma} I_{2n})\Big)^{-1}$, see \cite[Equation~3.2]{Mehrmann2012}. 
\end{remark}

\subsection{Two equivalent eigenvalue problems on an infinite-dimensional space} 
\label{subsec:infinite_dimensional}
In this subsection we examine the relation between the NLEVP \eqref{eq:nlevp} and two linear but infinite-dimensional eigenvalue problems. To this end, consider the space of continuous functions that map the interval $\left[-\tau_K,\tau_K\right]$ to $\C^{2n}$, denoted by $X:= C\big(\left[-\tau_K,\tau_K\right],\C^{2n}\big)$, and define the linear operator ${\operator:D(\operator)\subseteq X \mapsto X}$ as
\begin{equation}
\label{eq:definition_operator}
\operator\varphi(\theta) := \varphi'(\theta) \text{ for } \theta \in \left[-\tau_K,\tau_K\right],
\end{equation} 
in which the domain of this operator, $D(\operator)$, consists of the functions $\varphi$ in $X$ that are continuously differentiable and that fulfill the condition
\begin{equation}
\label{eq:integration_condition}
\varphi'(0) = H_0\varphi(0) + \sum_{k=1}^{K} \Big(H_{-k} \varphi(-\tau_k)+H_k \varphi(\tau_k)\Big),
\end{equation}
or in other words
\[
D(\operator) := \left\{ \varphi\in X: \varphi'\in X \  \& \ \varphi \ \text{satisfies} \ \eqref{eq:integration_condition}\right\}.
\]
It can be shown that the operator $\operator$ only features a point spectrum. Furthermore, a complex number $\lambda$ is an eigenvalue of this operator if (and only if) there exists a non-trivial function $\varphi \in D(\operator)$ such that
\begin{equation*}
\label{eq:infinite_dimensional_hamiltonian_evp}
\operator\varphi = \lambda \varphi .
\end{equation*}
This function $\varphi$ is called the eigenfunction of $\operator$ associated with the eigenvalue $\lambda$. From \eqref{eq:definition_operator}, it is clear that these eigenfunctions must have the form $ve^{\lambda \cdot}$ with $v\in\C^{2n}$. By plugging this result into \eqref{eq:integration_condition} it is evident that there exists a correspondence between a right eigenpair $(\lambda,v)$ of \eqref{eq:nlevp} and an eigenpair $(\lambda,\varphi)$ of $\operator$. This relation is stated more rigorously in the following proposition.
\begin{proposition}
	\label{proposition:equivalence_infinite_nonlinear}
	It holds that
	\begin{enumerate}
		\item if $(\lambda,v)$ is a right eigenpair of the NLEVP associated with \eqref{eq:characteristic_matrix}, then $(\lambda, v e^{\lambda \cdot})$ is eigenpair of $\operator$, and
		\item if $(\lambda,\varphi)$ is an eigenpair of $\operator$, then the eigenfunction $\varphi$ is of the form $ve^{\lambda \cdot}$ with ($\lambda$,$v$) a right eigenpair of the NLEVP associated with \eqref{eq:characteristic_matrix}.
	\end{enumerate}
\end{proposition}
\begin{proof}
	The proposition follows from a similar argument as  \cite[Proposition~2.2]{Michiels2010}.
\end{proof}

As seen in \Cref{proposition:equivalence_infinite_nonlinear}, the right eigenpairs of the considered NLEVP are connected with the eigenpairs of the operator $\operator$. Now we will look for an operator which has a similar connection with the left eigenpairs of the considered NLEVP. To this end, let us introduce another operator $\adjointoperator: D(\adjointoperator)\subseteq X\mapsto X$:
\[
\adjointoperator \psi(\theta) := - \psi'(\theta) \text{ for } \theta \in \left[-\tau_K,\tau_K\right],
\]
with
\[
D(\adjointoperator) := \left\{\psi \in X: \psi'\in X \,\&\, -\psi'(0) = H_0^{\top} \psi(0) + \textstyle\sum\limits_{k=1}^{K} \Big(H_k^{\top}\psi(-\tau_k) + H_{-k}^{\top} \psi(\tau_k)\Big) \right\}.
\]
The infinite-dimensional eigenvalue problem associated with this operator has the following relation with the left eigenpairs of the considered NLEVP.
\begin{proposition} It holds that,
	\begin{enumerate}
		\item if $(\lambda, w)$ is a left eigenpair of the NLEVP associated with \eqref{eq:characteristic_matrix}, then\linebreak $(\lambda, w e^{-\lambda \cdot})$ is an eigenpair of $\adjointoperator$.
		\item if $(\lambda,\psi)$ is an eigenpair of $\adjointoperator$, then the eigenfunction $\psi$ has the form $w e^{-\lambda \cdot}$ with $(\lambda,w)$ a left eigenpair of the NLEVP associated with \eqref{eq:characteristic_matrix}.
	\end{enumerate}
\end{proposition}
\begin{proof}
	As before, the assertions follow from a similar argument as in \cite[Proposition~2.2]{Michiels2010}.
\end{proof}

Combined with  \Cref{proposition:hamiltonian_symmetry}, these results imply that the following relation between the eigenpairs of $\operator$ and those of $\adjointoperator$ holds.
\begin{corollary}
	If $\big(\lambda,\varphi\big)$ is an  eigenpair of $\operator$, then \mbox{$\big(-\lambda,\psi\big)$}, with $\psi(\theta) = J\varphi(\theta)$ for $\theta \in [-\tau_K,\tau_K]$, is an eigenpair of $\adjointoperator$ and visa versa.
\end{corollary}

Furthermore, let us introduce the bilinear form $\bilinear{\cdot}{\cdot}: X \times X \mapsto \C$ with
\begin{multline}
\label{eq:bilinear_form}
\bilinear{\varphi}{\psi} := \psi(0)^{\top} \varphi(0)+ \\ \sum_{k=1}^{K} \left(\int_0^{\tau_k} \psi(\theta)^{\top} H_{-k} \varphi(\theta-\tau_k) \dd\theta - \int_0^{\tau_k} \psi(\theta-\tau_k)^{\top} H_k \varphi(\theta) \dd\theta \right).
\end{multline} 
Notice that this bilinear form does not define an inner product as it is neither Hermetian symmetric nor  positive definite. However, this bilinear form does induce two important relations between the operators $\operator$ and  $\adjointoperator$. Firstly, borrowing terminology from \cite[Chapters 7]{Hale1977}, $\adjointoperator$ can be seen as the \emph{formal adjoint} of $\operator$ with respect to the bilinear form \eqref{eq:bilinear_form}, since the following result holds.

\begin{proposition}
	For $\varphi\in D(\operator)$ and $\psi \in D(\adjointoperator)$ the equality \[ \bilinear{ \operator\varphi}{\psi} = \bilinear{\varphi} {\adjointoperator\psi}\] holds.
\end{proposition} 
\begin{proof} Using the definition of both operators and partial integration, we find that
	\begin{align*}
	\bilinear{\varphi}{\adjointoperator\psi}
	=& \, \begin{multlined}[t] \textstyle
	-\psi'(0)^{\top}\varphi(0) - \sum\limits_{k=1}^{K} \left(\int\limits_{0}^{\tau_k} \psi'(\theta)^{\top} H_{-k} \varphi(\theta-\tau_k) \dd\theta  - \qquad \quad\right. \\  \left. \textstyle \int\limits_0^{\tau_k} \psi'(\theta-\tau_k)^{\top} H_k \varphi(\theta) \dd\theta \right)
	\end{multlined} \\
	=& \,\psi(0)^{\top} \underbrace{\bigg(H_0 \varphi(0) + \textstyle\sum\limits_{k=1}^{K} \big(H_{-k} \varphi(-\tau_k) + H_k \varphi(\tau_k)\big)\bigg)}_{\varphi'(0)} + \\ & \qquad\qquad \textstyle \sum\limits_{k=1}^{K} \Big( \int\limits_{0}^{\tau_k} \psi(\theta)^{\top} H_{-k} \varphi'(\theta-\tau_k) \dd\theta 
	 - \int\limits_0^{\tau_k} \psi(\theta-\tau_k)^{\top} H_k \varphi'(\theta) \dd\theta \Big) \\
	=&\, \bilinear{\operator\varphi}{\psi}.
	\end{align*}
\end{proof}
Secondly, under the bilinear form \eqref{eq:bilinear_form} the eigenfunctions of $\adjointoperator$ are complementary to the eigenfunctions of $\operator$, as spelled out in the following proposition.
\begin{proposition}
	If $\varphi_{\lambda}$ is an eigenfunction of $\operator$ associated with an eigenvalue $\lambda$ and if $\psi_{\mu}$ is an eigenfunction of $\adjointoperator$ associated with a different eigenvalue $\mu\neq\lambda$ then $\bilinear{\varphi_{\lambda}}{\psi_{\mu}} = 0$.
\end{proposition} 
\begin{proof} It follows from the result above that $\bilinear{\operator\varphi_{\lambda}}{\psi_{\mu}} = \bilinear{\varphi_{\lambda}}{\adjointoperator \psi_{\mu}}$. Using the properties of a bilinear form one finds that
	\[
	\begin{aligned}
	0 &= \bilinear{\operator \varphi_{\lambda}}{\psi_{\mu}} - \bilinear{\varphi_{\lambda}}{\adjointoperator \psi_{\mu}} \\
	&=\bilinear{\lambda \varphi_{\lambda}}{\psi_{\mu}} - \bilinear{\varphi_{\lambda}}{\mu \psi_{\mu}} \\
	&= (\lambda-\mu)\,\bilinear{\varphi_{\lambda}}{\psi_{\mu}}.
	\end{aligned}
	\]
	As $\lambda - \mu \neq 0$, $\bilinear{\varphi_{\lambda}}{\psi_{\mu}}$ must equal zero.
\end{proof}
\section{The structure preserving shift-invert infinite Arnoldi method for Hamiltonian delay eigenvalue problems}
\label{sec:structure_preserving_delay}
In the previous section we saw that the eigenvalues of the NLEVP \eqref{eq:nlevp} correspond to those of the linear but infinite-dimensional operator $\operator$. To compute the eigenvalues of this operator in the neighborhood of some purely real or imaginary shift $\sigma$, an extension of the infinite Arnoldi method for RDEVPs, proposed in \cite{Jarlebring2010}, will be introduced in this section. The presented method differs from the traditional shift-invert infinite Arnoldi method in that it preserves the special Hamiltonian structure of the spectrum. To this end, as was the case in \Cref{subsec:finite_dimensional_case}, we  first have to derive an adequate shift-invert transformation.  Inspired by \eqref{eq:shift_invert_matrix}, lets us consider the following linear but infinite-dimensional operator:
\begin{equation*}
\mappedoperator := \left(\operator-\sigma\mathcal{I}_X\right)\left(\operator+\sigma\mathcal{I}_X\right),
\end{equation*}
with $\mathcal{I}_X$ the identity operator on $X$.
By plugging in the definition of $\operator$ this implies that
\begin{equation*}
\mappedoperator \varphi(\theta) = \varphi''(\theta)-\sigma^2\varphi(\theta) \text{ for } \theta \in \left[-\tau_K,\tau_{K}\right] \text{ and } \varphi \in D(\mappedoperator),
\end{equation*}
with $D(\mappedoperator)$, the domain of $\mappedoperator$, consisting of the functions $\varphi\in X$ which are twice continuously differentiable and which fulfil the following two conditions
\begin{align}
\varphi'(0) &= H_0 \varphi(0) + \textstyle\sum\limits_{k=1}^{K} \big(H_{-k} \varphi(-\tau_k) + H_k \varphi(\tau_k)\big) \text{ and } \label{eq:structure_preserving_operator_D1} \\  \varphi''(0) &= H_0 \varphi'(0) + \textstyle\sum\limits_{k=1}^{K} \big(H_{-k} \varphi'(-\tau_k) + H_{k} \varphi'(\tau_k)\big), \label{eq:structure_preserving_operator_D2}
\end{align}
or in other words \[
D\left(\mappedoperator \right) = \left\{\varphi \in X: \varphi'\in X\ , \ \varphi'' \in X \ \& \ \varphi \text{ satisfies \eqref{eq:structure_preserving_operator_D1} and \eqref{eq:structure_preserving_operator_D2}}\right\}.\]
 If $\sigma$ is not an eigenvalue of $\operator$, then the operator $\mappedoperator$ is invertible.
For $\sigma\neq0$ one finds:
\begin{equation}
\label{eq:solution_inverse_operator}
\mappedoperator^{-1}\phi\,(\theta)  = \textstyle \left(\int\limits_0^{\theta} \frac{\phi(\eta)}{2\sigma} e^{-\sigma \eta} d\eta +\mathrm{C}_{\sigma}[\phi]\right) e^{\sigma \theta} + \left(-\int\limits_0^{\theta} \frac{\phi(\eta)}{2\sigma} e^{\sigma \eta} d\eta + \mathrm{C}_{-\sigma}[\phi]\right) e^{-\sigma \theta},
\end{equation}
in which the constants $\mathrm{C}_{\sigma}[\phi]$ and $\mathrm{C}_{-\sigma}[\phi]$ are uniquely defined by conditions \eqref{eq:structure_preserving_operator_D1} and \eqref{eq:structure_preserving_operator_D2}. After some straight forward manipulations
 we find
\begin{align}
2 \sigma M(\sigma)\ \mathrm{C}_{\sigma}[\phi] &= \negative \phi(0) +  \textstyle \sum\limits_{k=1}^{K} \left[H_k \int\limits_{0}^{\tau_k} \phi(\eta) e^{-\sigma (\eta-\tau_k)} \dd\eta-H_{\negative k} \int\limits_{0}^{\tau_k} \phi(\eta-\tau_k) e^{-\sigma \eta} \dd\eta  \right] \label{eq:const_plus}
\intertext{and} \textstyle
2\sigma M(-\sigma)\ \mathrm{C}_{\negative\sigma}[\phi] &= \phantom{\negative}  \phi(0) - \textstyle \sum\limits_{k=1}^{K} \left[H_k \int\limits_{0}^{\tau_k} \phi(\eta)e^{\sigma (\eta-\tau_k)\phantom{-}} \dd\eta  - H_{\negative k} \int\limits_{0}^{\tau_k} \phi(\eta-\tau_k) e^{\sigma\eta\phantom{-}} \dd\eta \right]\!,\label{eq:const_min}
\end{align}
with $M(\cdot)$ the characteristic matrix as introduced in \eqref{eq:characteristic_matrix}. Note that to determine the integration constants $C_{\sigma}[\phi]$ and $C_{-\sigma}[\phi]$, one has to solve a linear system with both $M(\sigma)$ and $M(-\sigma)$ to preserve the symmetries in the spectrum. For $\sigma = 0$, one finds
\begin{equation}
\label{eq:solution_inverse_operator0}
\mathcal{R}_0^{-1}\phi\,(\theta) = \int_0^{\theta}\int_0^{\eta_2} \phi(\eta_1) \dd \eta_1 \dd \eta_2 + \mathrm{C}_1[\phi] \theta + \mathrm{C}_0[\phi], 
\end{equation}
in which $\mathrm{C}_1[\phi]$ and $\mathrm{C}_0[\phi]$ again follow from \eqref{eq:structure_preserving_operator_D1} and \eqref{eq:structure_preserving_operator_D2}:
\begin{align*}
M(0)\ \mathrm{C}_1[\phi] &= -\phi(0) + \textstyle\sum\limits_{k=1}^{K} \left(H_k \int\limits_0^{\tau_k} \phi(\eta)\dd\eta + H_{-k} \int\limits_{0}^{-\tau_k} \phi(\eta) \dd\eta \right), \\
M(0)\ \mathrm{C}_0[\phi] &= -\mathrm{C}_1[\phi] + \textstyle\sum\limits_{k=1}^{K}\left[H_k\left( \int\limits_{0}^{\tau_k}  \int\limits_{0}^{\eta_2} \phi(\eta_1) \dd\eta_1 \dd\eta_2 + \mathrm{C}_1[\phi]\tau_k\right) \right. +\\ 
&\phantom{=} \textstyle \left. H_{-k}\left( \int\limits_{0}^{-\tau_k}  \int\limits_{0}^{\eta_2} \phi(\eta_1) \dd\eta_1 \dd\eta_2 - \mathrm{C}_1[\phi]\tau_k\right)\right].
\end{align*}

Next, we state some important properties of this operator $\mappedoperator^{-1}$.

\begin{proposition}
	\label{prop:real_valued_Krylov_subspace}
	If $\phi \in X$ is a real-valued function and $\sigma$ is purely real or purely imaginary and not an eigenvalue of $\mathcal{H}$, then $\mappedoperator^{-1}\phi$ is a real function. 
\end{proposition}
\begin{proof}
	For $\sigma\neq 0$, this assertion follows directly from \eqref{eq:solution_inverse_operator}, \eqref{eq:const_plus} and \eqref{eq:const_min} by noting that $\mathrm{C}_{\sigma}[\phi]$ and $\mathrm{C}_{-\sigma}[\phi]$ are real if $\sigma$ is real and that $\mathrm{C}_{-\sigma}[\phi]$  is the complex conjugate of $\mathrm{C}_{\sigma}[\phi]$ if $\sigma$ is purely imaginary. For $\sigma=0$ this result follows from \eqref{eq:solution_inverse_operator0} and the expressions for  $C_1[\phi]$ and $C_0[\phi]$ given above.
\end{proof}
\begin{lemma}
	\label{property:skew_hamiltonian}
	For $\sigma$ not an eigenvalue of $\mathcal{H}$, the operator $\mappedoperator^{-1}$ is self-adjoint with respect to $\bilinear{\cdot}{J\cdot}$ meaning that the equality $\bilinear{\mappedoperator^{-1} \varphi}{J\psi} = \bilinear{\varphi}{ J{\mappedoperator}^{-1}\psi}$ holds.
\end{lemma}
\begin{proof}
	Here we restrict ourselves to the case $\sigma\neq 0$ as the proof for $\sigma=0$ is similar. \\
	Using the definition of the bilinear form $\bilinear{\cdot}{\cdot}$ given in \eqref{eq:bilinear_form}, and using \eqref{eq:solution_inverse_operator} we find
	\begin{align*}
		\textstyle
	\bilinear{\mappedoperator^{-1}\varphi}{J\psi} =&   -\Bigg[ \psi(0)^{\top}J\Big(\mathrm{C}_{\sigma}[\varphi] + \mathrm{C}_{-\sigma}[\varphi]\Big) + \\ & 	\textstyle
	\sum\limits_{k=1}^{K} \bigg( \int\limits_{0}^{\tau_k}  \psi(\theta)^{\top} JH_{\negative k}  \bigg[\bigg(\mathrm{C}_{\sigma}[\varphi]+ \int\limits_0^{\theta\negative\tau_k} \frac{\varphi(\eta)}{2\sigma} e^{\negative\sigma \eta} \dd\eta   \bigg) e^{\sigma (\theta\negative\tau_k)} + \\ 	 & \textstyle \qquad\qquad\qquad\qquad\qquad
	\bigg( \mathrm{C}_{\negative\sigma}[\varphi] -\!\!\int\limits\limits_0^{\theta\negative\tau_k} \! \frac{\varphi(\eta)}{2\sigma} e^{\sigma \eta} \dd\eta \bigg) e^{\negative\sigma (\theta\negative\tau_k)} \bigg] \dd\theta  \\ &   	\textstyle
	- \int\limits_{0}^{\tau_k} \psi(\theta-\tau_k)^{\top} JH_{k} \bigg[\bigg(\mathrm{C}_{\sigma}[\varphi]+\int\limits_0^{\theta} \frac{\varphi(\eta)}{2\sigma} e^{-\sigma \eta} \dd\eta  \bigg) e^{\sigma \theta} +   \\ &	\textstyle \qquad \qquad \qquad \qquad \qquad \qquad
	\bigg( \mathrm{C}_{\negative \sigma}[\varphi]-\int\limits_0^{\theta} \frac{\varphi(\eta)}{2\sigma} e^{\sigma \eta} \dd\eta  \bigg) e^{-\sigma \theta} \bigg]\dd\theta \Bigg)\Bigg]. 
	\end{align*}
	
	Firstly, observe that
	\begin{align*}
	\textstyle
	\int\limits_{0}^{\tau_k} \psi(\theta)^{\top} JH_{\negative k} \int\limits_0^{\theta-\tau_k} \frac{\varphi(\eta)}{2\sigma} e^{\negative\sigma \eta} \dd\eta &e^{\sigma (\theta-\tau_k)} \dd\theta \\ & =\textstyle \int\limits_{0}^{-\tau_k} \int\limits_{0}^{\tau_{k}+\eta} \psi(\theta)^{\top} e^{\sigma(\theta-\tau_k)} \dd\theta JH_{-k} \frac{\varphi(\eta)}{2\sigma} e^{\negative \sigma \eta} \dd\eta \\
	&\textstyle = \int\limits_{\tau_k}^{0} \int\limits_{0}^{\hat{\eta}} \psi(\theta)^{\top} e^{\sigma(\theta-\tau_k)} \dd\theta JH_{\negative k} \frac{\varphi(\hat{\eta}-\tau_k)}{2\sigma} e^{\negative\sigma (\hat{\eta}-\tau_k)} \dd\hat{\eta} \\
	&\textstyle = \negative \int\limits_{0}^{\tau_k} \int\limits_{0}^{\hat{\eta}} \frac{\psi(\theta)}{2\sigma}^{\top} e^{\sigma\theta} \dd\theta e^{-\sigma \hat{\eta}} JH_{\negative k} \varphi(\hat{\eta}-\tau_k)  \dd\hat{\eta},
	\end{align*}

	in which we first interchanged the integrating variables, then made the change of variable $\hat{\eta} = \tau_k + \eta$ and finally interchanged the integration boundaries and reordered the terms.
	Using the same procedure we find:
	\begin{equation*}
	\textstyle
	-\int\limits_{0}^{\tau_k} \psi(\theta)^{\top} JH_{-k} \!\! \int\limits_{0}^{\theta-\tau_k} \! \frac{\varphi(\eta)}{2\sigma} e^{\sigma \eta} \dd\eta\ e^{-\sigma(\theta-\tau_k)} \dd\theta  =  \int\limits_{0}^{\tau_k} \int\limits_{0}^{\hat{\eta}} \frac{\psi(\theta)}{2\sigma}^{\top} e^{-\sigma\theta} \dd\theta e^{\sigma\hat{\eta}} JH_{-k} \varphi(\hat{\eta}-\tau_k)  \dd\hat{\eta}.
	\end{equation*}
	In a similar fashion we obtain:
	\begin{align*}
	\textstyle
	\negative\int\limits_{0}^{\tau_k} \psi(\theta\minus\tau_k)^{\top} JH_k \int\limits_0^{\theta} \frac{\varphi(\eta)}{2\sigma} e^{\negative\sigma \eta} \dd\eta e^{\sigma \theta} \dd\theta &= \textstyle \int\limits_0^{\tau_k} \int\limits_{0}^{\eta\minus\tau_k} \frac{\psi(\hat{\theta})}{2\sigma}^{\top}\! e^{\sigma\hat{\theta}} \dd \hat{\theta} e^{\negative\sigma(\eta\minus\tau_k) }JH_k \varphi(\eta) \dd\eta \intertext{ and } 
	\textstyle \int\limits_{0}^{\tau_k} \psi(\theta\minus\tau_k)^{\top} JH_k \int\limits_0^{\theta} \frac{\varphi(\eta)}{2\sigma}e^{\sigma \eta} \dd\eta e^{\negative\sigma \theta} \dd\theta & \textstyle=\negative \int\limits_0^{\tau_k} \int\limits_{0}^{\eta\minus\tau_k} \frac{\psi(\hat{\theta})}{2\sigma}^{\top}\! e^{\negative\sigma \hat{\theta}} \dd\hat{\theta} e^{\sigma(\eta\minus\tau_k)} JH_k \varphi(\eta) \dd\eta,
	\end{align*}
	in which we used the change of variables $\hat{\theta} = \theta - \tau_k$.
	Next, by grouping the terms with $\mathrm{C}_{\sigma}[\varphi]$ and using $J^{\top} = - J$, $JH_{-k} = (JH_{k})^{\top}$ and $J\big(2\sigma M(\sigma)\big)^{-1} = -\big(2\sigma M(-\sigma)\big)^{-T}J$ we find
	\begin{multline*}
	\textstyle \left(\psi(0)^{\top} J + \sum\limits_{k=1}^{K} \left[\int\limits_0^{\tau_k} \psi(\theta)^{\top} J H_{-k} e^{\sigma(\theta-\tau_k)} \dd\theta - \int\limits_0^{\tau_k} \psi(\theta-\tau_k)^{\top}J H_k e^{\sigma \theta} \dd\theta \right] \right) C_{\sigma}[\varphi] \\
	\textstyle = \left( \psi(0)^{\top}- \sum\limits_{k=1}^{K} \left[\int\limits_{0}^{\tau_k}\psi(\theta)^{\top} H_k^{\top} e^{\sigma(\theta-\tau_k)}\dd\theta - \int\limits_0^{\tau_k}\psi(\theta-\tau_k)^{\top} H_{-k}^{\top} e^{\sigma \theta} \dd\theta \right]\right) J \mathrm{C}_{\sigma}[\varphi] \\
	\textstyle = \mathrm{C}_{-\sigma}[\psi]^{\top} J\left(\varphi(0) -  \sum\limits_{k=1}^{K} \Big[\negative H_{-k} \int\limits_{0}^{\tau_k} \varphi(\eta-\tau_k) e^{-\sigma \eta} \dd\eta  + H_k \int\limits_{0}^{\tau_k} \varphi(\eta) e^{-\sigma( \eta-\tau_k)} \dd\eta \Big]\right).
	\end{multline*}
	Similarly for the terms with $C_{-\sigma}\left[\varphi\right]$ we find
	\begin{multline*}
	\textstyle
	\left(\psi(0)^{\top}J + \sum\limits_{k=1}^{K}\left[ \int\limits_0^{\tau_k} \psi(\theta)^{\top} J H_{-k} e^{-\sigma(\theta-\tau_k)} \dd\theta - \int\limits_0^{\tau_k} \varphi(\theta-\tau_k)^{\top} J H_k e^{-\sigma \theta} \dd\theta\right]\right) \mathrm{C}_{-\sigma}[\varphi] \\ \textstyle = C_{\sigma}[\psi]^{\top}J \left(\varphi(0)+\sum\limits_{k=1}^{K}\left[H_{-k}\int\limits_0^{\tau_k} \varphi(\eta-\tau_k) e^{\sigma \eta} \dd\eta - H_k \int\limits_0^{\tau_k}\varphi(\eta)e^{\sigma(\eta-\tau_k)} \dd \eta\right]\right).
	\end{multline*}
	Combining all these relations, we obtain the desired result.
\end{proof}
Next, we introduce some properties of the (infinite-dimensional) eigenvalue problem associated with this operator:
\begin{equation}
\label{eq:mapped_eigenvalue_problem}
\mappedoperator^{-1} \phi = \mu \phi.
\end{equation}
\begin{proposition} For $\sigma$ not an eigenvalue of $\operator$, the infinite-dimensional eigenvalue problem in \eqref{eq:mapped_eigenvalue_problem} satisfies the following three properties
	\label{property:R_sigma_inverse}.
	\begin{enumerate}
		\item 	If $\lambda$ is an eigenvalue of $\operator$ then $\mu = \frac{1}{\lambda^2-\sigma^2}$ is an eigenvalue of $\mappedoperator^{-1}$ and visa versa, if $\mu$ is an eigenvalue of $\mappedoperator^{-1}$ then $\pm\sqrt{\frac{1}{\mu}+\sigma^2}$ are eigenvalues of $\operator$. 
		\item Let $\lambda\neq 0$ and $-\lambda$ be simple eigenvalues of $\mathcal{H}$ with eigenfunctions  $\varphi_{+}$ and $\varphi_{-}$, respectively, then the eigenspace of $\mu = \frac{1}{\lambda^2-\sigma^2}$ is spanned by $\left\{\varphi_{+},\varphi_{-}\right\}$. 		
		\item The eigenvalues of $\mappedoperator^{-1}$ have even multiplicity.
		
	\end{enumerate}
\end{proposition}
\begin{proof}
	First we will show that if $\varphi\in D(\operator)$ is an eigenfunction of $\operator$ associated with $\lambda$, \ie $\operator\varphi = \lambda \varphi$, then this function is also an eigenfunction of $\mappedoperator^{-1}$  associated with $\mu = \frac{1}{\lambda^2-\sigma^2}$. This result follows immediately from the following equality:
	\[
	\mappedoperator\varphi = (\operator - \sigma I_{X})(\operator + \sigma I_{X}) \varphi = (\lambda+\sigma)(\operator - \sigma I_{X})\varphi =  (\lambda^{2}-\sigma^2)\varphi,
	\]
	which implies that $\mappedoperator^{-1} \varphi = \frac{1}{\lambda^2- \sigma^2}\varphi$. Visa versa, if there exists a $\phi$ such that $\mappedoperator^{-1}\phi = \mu \phi$, then $\left(\mathcal{H}-\sqrt{\frac{1}{\mu}+\sigma^2}\mathcal{I}_{X}\right)\left(\mathcal{H}+\sqrt{\frac{1}{\mu}+\sigma^2}\mathcal{I}_{X}\right)\phi = 0$, meaning that either $\mathcal{H}-\sqrt{\frac{1}{\mu}+\sigma^2}\mathcal{I}_{X}$ or $\mathcal{H}+\sqrt{\frac{1}{\mu}+\sigma^2}\mathcal{I}_{X}$ has a non-trivial null space. The fact that both $\pm \sqrt{\frac{1}{\mu}+\sigma^2}$ are  eigenvalues of $\operator$ follows from \Cref{proposition:hamiltonian_symmetry}.
	
	Thus if $\lambda\neq0$ is an eigenvalue of $\operator$, then both $\lambda$ and $-\lambda$ are mapped together to $\mu = \frac{1}{\lambda^2- \sigma^2}$ and the eigenspace associated with $\mu$ is spanned by their eigenfunctions. 
	
	If $\lambda = 0$ is an eigenvalue of $\operator$, then it has even multiplicity. As a consequence, $-\sigma^{-2}$ is an eigenvalue of $\mappedoperator^{-1}$ with the same multiplicity.
\end{proof}

Next we examine the following Krylov subspace generated by $\mappedoperator^{-1}$:
\begin{equation*}
\mathrm{K}_m\left(\mappedoperator^{-1},\phi_1\right) = \Span\left\{\phi_1, \mappedoperator^{-1} \phi_1, \mappedoperator^{-2}\phi_1,\dots,\mappedoperator^{-(m-1)}\phi_1 \right\} \subseteq X,
\end{equation*}
with $\phi_1\in X$ an arbitrary real-valued starting function. This subspace has the following important property.

\begin{proposition}
	\label{theorem:J_neutrality} Let $\sigma$ not be an eigenvalue of $\mathcal{H}$ and let $\phi_1\in X$ be a real-valued function, then for any two functions $\varphi$ and $\psi$ in the Krylov subspace $\mathrm{K}_m\left(\mappedoperator^{-1},\phi_1\right)$ the equality $\bilinear{\varphi}{J\psi} = 0$ holds, with $\bilinear{\cdot}{\cdot}$ as defined in \eqref{eq:bilinear_form}.
\end{proposition}
\begin{proof}
	Each function in $\mathrm{K}_{m}\left(\mappedoperator^{-1},\phi_1\right)$ can be written as $P_{m-1}(\mappedoperator^{-1})\phi_1$ with $P_{m-1}(\cdot)$ a polynomial of degree (at most) $m-1$. Due to the bilinearity of $\bilinear{\cdot}{J\cdot}$, it thus suffices to prove that $\bilinear{\mappedoperator^{-i}\phi_1}{J\mappedoperator^{-j}\phi_1} = 0$ for $i,j=0,\dots,m-1$.
	
	By \Cref{property:skew_hamiltonian} we have
	\begin{equation*}
	\bilinear{\mappedoperator^{-i}\phi_1}{J \mappedoperator^{-j}\phi_1} 	=\bilinear{\mappedoperator^{-(j+i)}\phi_1}{J \phi_1}= \bilinear{\phi_1}{J\mappedoperator^{-(j+i)}\phi_1}.
	\end{equation*}
	However, due to the fact that $\bilinear{\varphi}{J\psi} = -\bilinear{\psi}{J\varphi}$, \ie{} $\bilinear{\cdot}{J\cdot}$ is anti-symmetric (see Appendix~\ref{appendix:anti_symmetry_bilinear_J_form}), we also have 
	\[
	\bilinear {\mappedoperator^{-(j+i)}\phi_1}{J\phi_1} = - \bilinear{ \phi_1}{J \mappedoperator^{-(j+i)}\phi_1},
	\]
	which implies that $\bilinear{\mappedoperator^{-i}\phi_1}{ J\mappedoperator^{-j}\phi_1} = 0$.
\end{proof}
This proposition has the following important consequence.
\begin{theorem}
	\label{theorem:uniqueness}
	For $\sigma$ not an eigenvalue of $\mathcal{H}$ and $\phi_1\in X$ a real-valued function, let $\lambda\neq0$ be a simple eigenvalue of $\mathcal{H}$ (or equivalently of the NLEVP \eqref{eq:nlevp}) with associated eigenfuction $\varphi_{+}$ and let $\varphi_{-}$ be the eigenfunction associated with the eigenvalue $-\lambda$, then the dimension of the intersection of $\Span \{\varphi_{+},\varphi_{-}\}$ and $ \mathrm{K}_m(\mappedoperator^{-1},\phi_1)$ is at most 1.
\end{theorem}
\begin{proof}
	Recall from \Cref{proposition:equivalence_infinite_nonlinear} that $\varphi_{+}(\theta)=v_{+} e^{\lambda \theta}$ and $\varphi_{-}(\theta) = v_{-}e^{-\lambda \theta}$, with $v_{+}$ and $v_{-}$ right eigenvectors of the NLEVP \eqref{eq:nlevp} associated with $\lambda$ and $-\lambda$, respectively. Furthermore, from \Cref{proposition:hamiltonian_symmetry} it follows that there exists a vector $w_{+}$ which is a left eigenvector of \eqref{eq:characteristic_matrix} associated with $\lambda$, such that $v_{-} = Jw_{+}$. Combing these properties with the definition of the bilinear form in \eqref{eq:bilinear_form}, we find that 
	\begin{align*}
	\bilinear{\varphi_{+}}{J\varphi_{\negative}} &= -\bigg[w_{+}^{\top} v_{+} + \sum_{k=1}^{K}\bigg(\int\limits_{0}^{\tau_k}e^{-\lambda \theta} w_{+}^{\top} H_{-k}v_{+} e^{\lambda(\theta-\tau_k)} \dd \theta\\
	& \qquad \qquad \qquad \qquad \quad \qquad \qquad \qquad - \int\limits_{0}^{\tau_k} e^{-\lambda(\theta-\tau_k)} w_{+}^{\top}  H_k v_{+} e^{\lambda\theta}\dd\theta \bigg)\bigg] \\
	&= -w_{+}^{\top}\left[I+\sum_{k=1}^{K}\Big(H_{-k}e^{-\lambda\tau_k}\tau_k-H_{k} e^{\lambda\tau_k} \tau_k\Big)\right]v_{+} = -w_{+}^{\top}\frac{\partial M(\lambda)}{\partial \lambda} v_{+} \neq 0,
	\end{align*}
	in which the last inequality follows from the fact that $\lambda$ is a simple eigenvalue of the NLEVP \eqref{eq:nlevp}. Thus if the intersection of $\Span\{\varphi_{+},\varphi_{-}\}$ and $ \mathrm{K}_m(\mappedoperator^{-1},\phi_1)$ has dimension two, then both $\varphi_{+}$ and $\varphi_{-}$ must lie inside this Krylov subspace, which would contradict \Cref{theorem:J_neutrality}.
\end{proof}
Combining this result with \Cref{property:R_sigma_inverse}, we have the following important result.
\begin{corollary}
	\label{corollary:projection_once}
	Under the conditions of \Cref{theorem:uniqueness} and assuming that $\lambda\neq0$ is a simple eigenvalue of $\operator$, the dimension of the intersection of the eigenspace of $\mappedoperator^{-1}$ associated with $\mu = 1/(\lambda^2-\sigma^2)$ and the Krylov subspace $\mathrm{K}_m\left(\mappedoperator^{-1},\phi_1\right)$ is at most 1.
\end{corollary}

We conclude this section with \Cref{alg:structure_preserving_shift_invert_arnoldi}, which gives a high-level description of the structure preserving shift-invert infinite Arnoldi method operating on the infinite-dimensional space $X$. In this algorithm  $\Psi_{[i,j]}$ denotes the element on the $i$\textsuperscript{th} row and $j$\textsuperscript{th} column of $\Psi$, $\Psi_{[1:m,:]}$ is the submatrix of $\Psi$ consisting of its first $m$ rows, $\langle\cdot,\cdot \rangle_X$ is an appropriate inner product on $X$ (which we will define later on), and $\|\cdot\|_X$ is the norm associated with this inner product.

The algorithm works as follows. Starting from a real initial function $\phi_1$, an orthonormal basis for the Krylov subspace $\mathrm{K}_m(\mappedoperator^{-1},\phi_1)$ is constructed iteratively. After obtaining this orthonormal basis, the eigenvalues of $\mappedoperator^{-1}$ are approximated by the eigenvalues of its orthogonal projection onto the Krylov subspace, \ie the submatrix $\Psi_{[1:m,:]}$, which has a reduced Hessenberg structure. Finally, the approximations for the eigenvalues of $\operator$ can be obtained using the transformation from \Cref{property:R_sigma_inverse}. Each iteration in \Cref{alg:structure_preserving_shift_invert_arnoldi} consists of the following steps. First the Krylov subspace is extended with the candidate function $\varphi_{i+1} = \mappedoperator^{-1}\phi_i$. Next, this function is orthogonalized against the already obtained basis vectors and subsequently normalized to have norm 1. Note that in contrast to the traditional infinite Arnoldi method, one now has to ensure that $\phi_{i+1}$ is also \emph{orthogonal} against $\phi_1,\dots,\phi_{i}$ with respect to $\bilinear{\cdot}{J\cdot}$. Although this last orthogonality constraint holds by construction in exact arithmetic (see \Cref{theorem:J_neutrality}), it is no longer the case when working in finite precision, as we shall illustrate in \Cref{subsec:example1}.

 Inspired by \cite{Jarlebring2010}, we define the following inner product on $X$ based on the scaled Chebyshev expansion. More precisely, let $\phi \in X$ and $\psi \in X$ be Lipschitz continuous functions, then for each of these functions there exists a unique Chebyshev expansion series that is absolute and uniformly convergent \cite{Trefethen2013}:
\[
\textstyle
\varphi(\theta) = \sum\limits_{l=0}^{\infty} c_l T_{l}\left(\frac{\theta}{\tau_K}\right)\text{ and } \psi(\theta) = \sum\limits_{l=0}^{\infty} d_l T_{l}\left(\frac{\theta}{\tau_K}\right) \text{ for } \theta \in [-\tau_K,\tau_K],
\]
with $T_l(\cdot):= \cos\big(l \arccos(\cdot)\big)$ the $l$\textsuperscript{th} Chebyshev polynomial of the first kind.
The inner product of these two functions is then defined as
\begin{equation}
\label{eq:inner_product}
\textstyle
\left\langle \psi,\varphi \right\rangle_X  := \sum_{l=0}^{\infty} c_l^{H} d_l.
\end{equation}	
Furthermore, it can be shown, see \cite[Equation 4.10]{Jarlebring2010}, that this inner product is equivalent with the following expression
\begin{equation*}
\textstyle
\left\langle \psi,\varphi \right\rangle_X  = \frac{2}{\pi}I\left[\varphi^{H}\psi \right] - \frac{1}{\pi^2}I[\varphi^{H}]I[\psi].
\end{equation*}
with $
I[f] = \int\limits_{-\tau_K}^{\tau_K} \frac{f(\theta)}{\sqrt{\tau_K^2-\theta^2}}\dd\theta.
$

As in the finite-dimensional case (see \Cref{subsec:finite_dimensional_case}), \Cref{theorem:uniqueness} and \Cref{corollary:projection_once} have an important consequence for the obtained approximations of purely imaginary eigenvalues of \eqref{eq:nlevp}. More specifically, since $\phi_1$ is chosen real-valued, \Cref{prop:real_valued_Krylov_subspace} and the definition of the inner product in \eqref{eq:inner_product}, imply that 	$\mathrm{K}_m\left(\mappedoperator^{-1},\phi_1\right)$ consists of real-valued functions and that the matrix $\Psi$ is real-valued. Now, let $\pm \jmath\omega$ be simple eigenvalues of $\mathcal{H}$. These two eigenvalues are mapped together to the purely real eigenvalue $\mu = \frac{1}{-\omega^2-\sigma^2}$ of $\mappedoperator^{-1}$. However, as long as the Krylov subspace satisfies \Cref{theorem:J_neutrality} (which is always the case in exact arithmetic), only a single component of the two-dimensional eigenspace associated with $\mu$ will be approximated in the Krylov subspace and hence the multiple eigenvalue $\mu$ will be approximated by a simple, real Ritz value in the reduced Hessenberg matrix $\Psi[1:m,:]$. After re-transformation the obtained approximations for $\pm\jmath\omega$ are thus typically purely imaginary. In contrast, if orthogonality with respect to $\bilinear{\cdot}{J\cdot}$ is lost (due to rounding errors), the Krylov subspace will eventually also approximate a second component in the eigenspace of $\mu$. As a consequence, the double eigenvalue $\mu$ will be approximated by a pair of Ritz values: either two nearby real Ritz values or a pair of complex conjugate Ritz values. In the latter case, after re-transformation the obtained approximations for the purely imaginary eigenvalues $\pm\jmath\omega$ of $\mathcal{H}$ thus have a small real part.

\begin{algorithm}[H]
	\caption{High-level description of the structure preserving shift-invert infinite Arnoldi method for NLEVP \eqref{eq:nlevp}.}
	\label{alg:structure_preserving_shift_invert_arnoldi}
	\begin{algorithmic}
		\STATE Choose a \textbf{real-valued} initial function $\phi_1$ with $\|\phi_1\|_X= 1$.
		\STATE Define a $(m+1) \times m$ matrix $\Psi$ with all elements equal to zero.
		\FOR{$i=1, \dots, m$}
		\STATE \textbf{Extend the Krylov subspace with} $\varphi_{i+1}=\mappedoperator^{-1} \phi_{i}$.
		\STATE \textbf{Orthogonalize} $\varphi_{i+1}$ against the already obtained orthogonal basis for the Krylov subspace and normalize the result:
		\begin{equation*}
		\Psi_{[i+1,i]} \phi_{i+1} = \varphi_{i+1} - \Psi_{[1,i]}\, \phi_1 - \dots - \Psi_{[i,i]}\, \phi_{i}
		\end{equation*}
		with $\Psi_{[j,i]}\!=\! \left\langle \phi_{j},\varphi_{i+1} \right\rangle_X$ for $j\!=\!1,\ldots, i$  and $\Psi_{[i+1,i]}$ such that $\|\phi_{i+1}\|_{X} =1$.
		\STATE \textbf{Ensure} orthogonality of $\phi_{i+1}$ against $\phi_1,\dots,\phi_{i}$ with respect to $\bilinear{\cdot}{J\cdot}$.
		\ENDFOR
		\STATE \textbf{Compute the eigenvalues} $\mu$ of $\Psi_{[1:m,:]}$.
		\STATE\textbf{Return} the approximations for the eigenvalues of \eqref{eq:nlevp}: $\pm \sqrt{\frac{1}{\mu}+\sigma^2}$.
	\end{algorithmic}
\end{algorithm}
The algorithm introduced above is described in terms of functions and therefore not directly implementable using finite-dimensional operations. The following two sections will therefore describe how this algorithm can be implemented using finite-dimensional linear algebra operations and how one can ensure that the built up Krylov subspace satisfies \Cref{theorem:J_neutrality}, even in the presence of rounding errors.

\section{Numerical implementation for the shift  equal to zero}
\label{sec:shift_zero}
We start with the case $\sigma = 0$. First we will see that by choosing an appropriate initial function $\phi_1$, a natural finite-dimensional representation for the functions and operations in \Cref{alg:structure_preserving_shift_invert_arnoldi} appears. To this end, observe from \eqref{eq:solution_inverse_operator} that if $\mathcal{R}_0^{-1}$ is applied on a vector-valued polynomial of degree $N$, the result is a vector-valued polynomial of degree $N+2$. As a consequence, when choosing a vector-valued polynomial as initial function, the constructed Krylov subspace consists entirely of vector-valued polynomials. By choosing an appropriate basis for the space of polynomials, Algorithm~\ref{alg:structure_preserving_shift_invert_arnoldi} can thus be carried out using finite-dimensional operations on the coefficient vectors of these polynomials. Due to our choice for the inner product in  \eqref{eq:inner_product}, it makes sense to work with scaled Chebyshev functions. The extension step in \Cref{alg:structure_preserving_shift_invert_arnoldi} can now be computed using straightforward linear algebra operations, as demonstrated in the following theorem. 
\begin{theorem}
	\label{theorem:zero_shift}
	Let 
	\[\phi_i(\theta) = \sum_{l=0}^{N_i} q_l^{(i)} T_l\left(\frac{\theta}{\tau_K}\right)\] with $q_l^{(i)} \in \R^{2n}$ for $l=0,\dots,N_{i}$, then \[
	\varphi_{i+1}(\theta)= \mathcal{R}_0^{-1}\phi_i(\theta) = \sum_{l=0}^{N_i+2} v_l^{(i+1)} T_l\left(\frac{\theta}{\tau_K}\right),
	\] in which the coefficient vectors $v_2^{(i+1)},\dots,v_{N_i+2}^{(i+1)}$ follow from
	{\small
	\begin{equation}
	\label{eq:extension_sigma0}
	\setlength{\arraycolsep}{2pt}
	\begin{bmatrix}
	v_2^{(i+1)} & \cdots & v_{N_i+2}^{(i+1)}
	\end{bmatrix} = \tau_K^2\begin{bmatrix}
	q_0^{(i)} & \cdots & q_{N_i}^{(i)}
	\end{bmatrix}\!\! \begin{bmatrix}
	\frac{1}{4} & \\
	& \frac{1}{24} \\
	\frac{\negative 1}{6} & & \frac{1}{48} \\
	& \frac{\negative 1}{16} & & \frac{1}{80} \\
	\frac{1}{24} & & \frac{-1}{30} & & \frac{1}{120} \\
	& \ddots & & \ddots & & \ddots \\
	& \multicolumn{2}{r}{\frac{1}{4(N_i-1)(N_i-2)}}& \multicolumn{2}{r}{\frac{-1}{2(N_i-1)(N_i+1)}} & \multicolumn{2}{r}{\frac{1}{4(N_i+1)(N_i+2)}}
	\end{bmatrix},
	\end{equation}}%
\normalsize
	while $v_1^{(i+1)}$ and $v_0^{(i+1)}$ follow from \eqref{eq:structure_preserving_operator_D1} and \eqref{eq:structure_preserving_operator_D2}:
	\begin{align*}
	M(0)\, v_1^{(i+1)} &=\textstyle \sum\limits_{l=2}^{N_i+2} \left( H_0 T_{l}'(0) + \sum\limits_{k=1}^{K} \left(H_k T_l'\Big(\frac{\tau_k}{\tau_K}\Big) + H_{\negative k} T_l'\Big(\frac{\negative\tau_k}{\tau_K}\Big)\right)-I_{2n} \frac{T_l''(0)}{\tau_K}\right)v_l^{(i+1)} 
	\intertext{ and }
	M(0)\, v_0^{(i+1)} &= \textstyle \sum\limits_{l=1}^{N_i+2} \left( H_0 T_{l}(0) + \sum\limits_{k=1}^{K} \left(H_k T_l\Big(\frac{\tau_k}{\tau_K}\Big) + H_{\negative k} T_l\Big(\frac{\negative\tau_k}{\tau_K}\Big)\right)-I_{2n} \frac{T_l'(0)}{\tau_K}\right)v_l^{(i+1)}\!,
	\end{align*}
	with $M(\cdot)$  the characteristic matrix as defined in \eqref{eq:characteristic_matrix}.
\end{theorem}
\begin{proof}
	We refer to \Cref{appendix:extension_step_sigma0} for more details on this result. 
\end{proof}
Notice that for large $n$, the main computation cost of the extension step consists of solving the systems for $v_0^{(i+1)}$ and $v_1^{(i+1)}$. However, because the same matrix $M(0)$ is used in each iteration, its factorization needs only to be computed ones, which greatly reduces the computation time.

Secondly, due to the chosen initial function, also the orthogonalisation step can be carried out using standard linear algebra operations. More specifically, as a consequence of \Cref{theorem:zero_shift}, the functions encountered in the $i$\textsuperscript{th} orthogonalisation step of \Cref{alg:structure_preserving_shift_invert_arnoldi} can be written as
\begin{equation}
\label{eq:cheb_expansion}
\phi_{j}(\theta) = \sum_{l=0}^{N_j} q^{(j)}_{l} T_l\left(\frac{\theta}{\tau_K}\right) \text{ for $j=1,\dots,i+1$ and } \varphi_{i+1}(\theta) = \sum_{l=0}^{N_{i+1}} v_l^{(i+1)} T_l\left(\frac{\theta}{\tau_K}\right),
\end{equation} 
with $q_l^{(j)}$ and $v_l^{(i+1)}$ belonging to $\R^{2n}$.	By introducing $v^{(i+1)}$ as the vector containing the stacked coefficients of $\varphi_{i+1}$, $q^{(j,i+1)}$ the vector containing the stacked coefficients of $\phi_{j}$ padded with zeros to length $2n(N_{i+1}+1)$ and $Q_{(i,i+1)}$  the matrix whose columns consists of the vectors $q^{(j,i+1)}$ for $j=1,\dots,i$
the orthogonalisation step reduces to
\begin{equation}
\label{eq:orthogonalisation_naive}
\Psi_{[i+1,i]}q^{(i+1,i+1)} = v^{(i+1)} -Q_{(i,i+1)}Q_{(i,i+1)}^{\top}v^{(i+1)}
\end{equation}
with $\Psi_{[i+1,i]}$ a normalisation constant such that $\|q^{(i+1,i+1)}\|_2 = 1$.

Up till now, the presented approach closely resembles the classic infinite Arnoldi method from \cite{Jarlebring2010}, but with the degree of the polynomials in the Krylov subspace increasing with two instead of one in each iteration. Recall however that it is important that the build-up Krylov subspace satisfies the condition in \Cref{theorem:J_neutrality}. Although this condition holds automatically in exact arithmetic, rounding errors cause a loss of orthogonality when working in finite precision. In finite precision arithmetic one therefore needs an \emph{additional orthogonalisation step}, which ensures orthogonality with respect to $\bilinear{\cdot}{J\cdot}$. To this end, note that for $\{\phi_j\}_{j=1}^{i}$ and $\varphi_{i+1}$ in form \eqref{eq:cheb_expansion} the bilinear form $\bilinear{\cdot}{J\cdot}$ can be evaluated using matrix-vector operations:
\begin{equation*}
\bilinear{\phi_j}{J\varphi_{i+1}} = {v^{(i+1)}}^{\top} \bigg(S_{N_{i+1}}^0 \otimes J  
+ \textstyle\sum\limits_{k=1}^{K} \Big( S_{N_{i+1}}^{-k} \otimes \big(J H_{-k}\big) 
+
S_{N_{i+1}}^{k} \otimes \big( J H_{k} \big) \Big) 
\bigg)q^{(j,i+1)},
\end{equation*}
in which $\otimes$ is the Kronecker product, the matrices $S_{N_{i+1}}^0$, $S_{N_{i+1}}^{1}$, $S_{N_{i+1}}^{-1}$, \dots, $S_{N_{i+1}}^{K}$ and $S_{N_{i+1}}^{-K}$ belong to $\R^{(N_{i+1}+1)\times (N_{i+1}+1)}$  and the elements on position $(l_1,l_2)$ of $S_{N_{i+1}}^{0}$, $S_{N_{i+1}}^{-k}$ and $S_{N_{i+1}}^{k}$ are equal to 
$
-T_{l_1}(0) T_{l_2}(0)$, $-\int\limits_{0}^{\tau_k} T_{l_1}\Big(\frac{\theta}{\tau_K}\Big) T_{l_2}\Big(\frac{\theta-\tau_k}{\tau_K}\Big) \dd \theta$ and \linebreak $ \int\limits_0^{\tau_k} T_{l_1}\Big(\frac{\theta-\tau_k}{\tau_K}\Big) T_{l_2}\Big(\frac{\theta}{\tau_K}\Big) \dd \theta,
$
respectively. Furthermore, observe that the matrix $$S_{N_{i+1}} = S_{N_{i+1}}^0\otimes J + \sum\limits_{k=1}^{K} \big( S_{N_{i+1}}^{-k} \otimes J H_{-k}  + S_{N_{i+1}}^{k} \otimes JH_k\big)$$ is skew-symmetric, \ie $S_{N_{i+1}}^{\top} = -S_{N_{i+1}}$.

In the $i$\textsuperscript{th} iteration step, the vector $q^{(i+1,i+1)}$  must thus be orthogonal with respect to the columns of both $Q_{(i,i+1)}$ and $S_{N_{i+1}} Q_{(i,i+1)}$. Moreover, because \linebreak ${q^{(j,i+1)}}^{\top}S_{N_{i+1}}q^{(j,i+1)} = 0$ for $j=1,\dots,i$ due to the skew-symmetry of $S_{{N_{i+1}}}$ and  ${q^{(j_1,i+1)}}^{\top}S_{N_{i+1}}q^{(j_2,i+1)} = 0$ for $j_1,j_2=1,\dots,i$ due to the orthogonality conditions in the previous iteration, the product $Q_{(i,i+1)}^{\top}S_{N_{i+1}} Q_{(i,i+1)}$ is zero. Thus in order to achieve simultaneous orthogonality with respect to $Q_{(i,i+1)}$ and $S_{N_{i+1}} Q_{(i,i+1)}$ we will use the following expression instead of \eqref{eq:orthogonalisation_naive}:
\begin{multline}
\label{eq:orthogonalisation_zero}
\Psi_{[i+1,i]} q^{(i+1,i+1)} = v^{(i+1)} - Q_{(i,i+1)} Q_{(i,i+1)}^{\top} v^{(i+1)} \\ - S_{N_{i+1}}Q_{(i,i+1)} \big((S_{N_{i+1}} Q_{(i,i+1)})^{\top}S_{N_{i+1}}Q_{(i,i+1)}\big)^{-1}(S_{N_{i+1}} Q_{(i,i+1)})^{\top} v^{(i+1)},
\end{multline}
with $\Psi_{[i+1,i]}$ a normalisation factor such that $\|q^{(i+1,i+1)}\|_2 = 1$. By using this expression, the resulting Krylov subspace satisfies \Cref{theorem:J_neutrality} up to machine precision, even in the presence of rounding errors.

\begin{remark}
	Note that the left upper blocks of the matrices $S_{N_{i+1}}^0$, $S_{N_{i+1}}^{-k}$ and $S_{N_{i+1}}^{k}$ are equal to $S_{N_{i}}^0$, $S_{N_{i}}^{-k}$ and $S_{N_{i}}^{k}$, respectively. Furthermore, note that these matrices only depend on the delays and not on the system matrices. This observation implies for instance that in the single delay case ($K=1$), where the delay can be rescaled to one by the substitution $s\leftarrow \tau_1 s$, these matrices can be precomputed.
\end{remark}

\section{Numerical implementation for non-zero shift}
\label{sec:non_zero_shift}
In this subsection, it will be assumed that the shift $\sigma$ is purely imaginary, \ie $\sigma = \jmath \omega$. A similar result can be obtained when the shift $\sigma$ is purely real, as will be explained in \Cref{remark:real_sigma}. 

Similar to the case $\sigma =0$ we will derive a compact representation for the functions and operations in \Cref{alg:structure_preserving_shift_invert_arnoldi} by choosing an appropriate structure for the initial function $\phi_1$. More specifically, if
\begin{equation}
\label{eq:preserverd_structure}
\phi(\theta) = P_N(\theta) e^{\jmath \omega \theta} + \overline{P_N(\theta)} e^{-\jmath \omega \theta} \text{ for }\theta \in [-\tau_K,\tau_K],
\end{equation} with $P_N(\cdot) : [-\tau_K,\tau_K] \mapsto \C^{2n}$ an arbitrary complex vector-valued polynomial of degree $N$ and $\overline{P_N(\theta)}$ the complex conjugate of $P_N(\theta)$, then the function $\mappedoperator^{-1}\phi$ is given by
\begin{equation}
\label{eq:preserverd_structure2}
\left(\mappedoperator^{-1}\phi\right)(\theta) = Q_{N+1}(\theta) e^{\jmath \omega \theta} + \overline{Q_{N+1}(\theta)} e^{-\jmath \omega \theta} \text{ for }\theta \in [-\tau_K,\tau_K],
\end{equation}
in which $Q_{N+1}(\cdot)$ is a complex vector-valued polynomial of degree $N+1$. Using this representation, the extension step in \Cref{alg:structure_preserving_shift_invert_arnoldi} can thus again be expressed in terms of operations on the vector-valued coefficients of these polynomials. This is formalized in the following theorem.

\begin{theorem}
	\label{theorem:non_zero_shift}
	If $\phi_i$ is given by
	\begin{equation}
	\label{eq:approach1_phi}
	\phi_i(\theta) := \left(\sum_{l=0}^{N_i} q_{l}^{(i)}\, T_l\left(\frac{\theta}{\tau_K}\right)\right)e^{\jmath \omega \theta} +\left(\sum_{l=0}^{N_i} \overline{q_{l}^{(i)}}\, T_l\left(\frac{\theta}{\tau_K}\right)\right) e^{-\jmath \omega \theta} 
	\end{equation}
	then extension step can be preformed using basic linear algebra operations:
	\begin{equation}
	\label{eq:approach1_varphi}\textstyle
	\varphi_{i+1}(\theta) = \left(\sum\limits_{l=0}^{N_i+1} v_{l}^{(i+1)} \, T_l\Big(\frac{\theta}{\tau_K}\Big) \right) e^{\jmath \omega \theta} + \left(\sum\limits_{l=0}^{N_i+1} \overline{v_{l}^{(i+1)}} \, T_l\Big(\frac{\theta}{\tau_K}\Big)\right) e^{-\jmath \omega \theta},
	\end{equation}
	in which the coefficient vectors $v_1^{(i+1)},\dots,v_{N_{i}+1}^{(i+1)}$ are given by
	\begin{multline}
	\label{eq:update_formula_approach1}
	\setlength{\arraycolsep}{2pt}
	\begin{bmatrix}
	v_{1}^{(i+1)} & \dots & v_{N_{i}+1}^{(i+1)}
	\end{bmatrix} \left(2\jmath\omega \tau_K\begin{bmatrix}
	\frac{1}{4} &  &  & &  \\
	& \frac{1}{6} &  &   &  \\
	\frac{-1}{4} &  & \frac{1}{8} &  &  \\
	& \ddots &  & \ddots & & \\
	& \multicolumn{3}{c}{\frac{-1}{2N_i}} &  \multicolumn{3}{c}{\frac{1}{2(N_i+2)}}
	\end{bmatrix} + \begin{bmatrix}
	0& & &\\
	1 &0 & &\\
	& \ddots &\ddots & \\
	& & 1 & 0
	\end{bmatrix} \right) \\[3mm] = \begin{bmatrix}
	q_{0}^{(i)} & \dots & q_{N_i}^{(i)}
	\end{bmatrix}  \tau_K^{2}\begin{bmatrix}
	\frac{1}{4} &  &  &  &  &  \\
	& \frac{1}{24}& & & & & \\
	\frac{-1}{6} & & \frac{1}{48} & & &\\
	& \frac{-1}{16} & &  \frac{1}{80} & & & &\\
	\frac{1}{24} & & \frac{-1}{30} & & & \frac{1}{120} & & \\
	& \ddots & &\ddots & & & \ddots & & \\
	& \multicolumn{3}{c}{\frac{1}{4(N_i-1)(N_i-2)}} &  \multicolumn{3}{c}{\frac{-1}{2(N_i-1)(N_i+1)}} & \multicolumn{3}{c}{ \frac{1}{4(N_i+1)(N_i+2)}}
	\end{bmatrix}
	\end{multline}
	and $v_{0}^{(i+1)}$ can be found by solving the following system
	\begin{align*}
\label{eq:expression_c0}
\jmath\omega  \,M(\jmath\omega)\, v_{0}^{(i+1)} =& \,\,
(H_0-\jmath \omega I_{2n}) \textstyle\sum\limits_{l=1}^{N_i+1}  \left(\Re\Big(v_{l}^{(i+1)} \Big)\frac{T_l'(0)}{\tau_K}+\jmath\omega v_{l}^{(i+1)} T_l(0)\right)  +\\ &
\textstyle \sum\limits_{k=1}^{K} \Big[ H_{k} \sum\limits\limits_{l=1}^{N_i+1}\Big(\Re\big( v_{l}^{(i+1)} e^{\jmath \omega \tau_k}\big)/\tau_K T_l'(\frac{\tau_k}{\tau_K})+\jmath \omega v_{l}^{(i+1)}T_l(\frac{\tau_k}{\tau_K})e^{\jmath \omega \tau_k}\Big) +\\  &
H_{\negative k} \textstyle\sum\limits_{l=1}^{N_i+1}\Big( \Re \big( v_{l}^{(i+1)} e^{-\jmath \omega \tau_k}\big)/\tau_K T_{l}'(-\frac{\tau_k}{\tau_K})+\jmath\omega v_{l}^{(i+1)} T_l(-\frac{\tau_k}{\tau_K}) e^{-\jmath \omega \tau_k}   \Big)\Big]
- \\&
\textstyle \sum\limits_{l=0}^{N_i} \Re\big(q_l^{(i)}\big) T_l(0),
\end{align*}
	with $M(\cdot)$ the characteristic matrix as defined in \eqref{eq:characteristic_matrix}.
\end{theorem}
\begin{proof}
	For more details on deriving this result, see \Cref{appendix:extension_step_jomega}.
\end{proof}
\begin{remark}
	\label{remark:real_sigma}
	When the shift $\sigma$ is purely real, similar results can be obtained. More specifically, if $\phi$ is given by
	\[\phi(\theta) = P_N^{+}(\theta) e^{\sigma \theta} + P_N^{-}(\theta) e^{-\sigma \theta}
	\] with $P_N^{+}(\cdot)$ and $P_N^{-}(\cdot)$ arbitrary real vector-valued polynomials of degree $N$, then $\mappedoperator^{-1}\phi$ is given by 
	\[
	\big(\mappedoperator^{-1}\phi\big)(\theta) = Q_{N+1}^{+}(\theta) e^{\sigma \theta} + Q_{N+1}^{-}(\theta) e^{-\sigma \theta}
	\]
	with $Q_{N+1}^{+}(\cdot)$ and $Q_{N+1}^{-}(\cdot)$ real vector-valued polynomials of degree $N+1$.
\end{remark}
From a theoretical point of view, the result obtained above is appealing: similarly to the case $\sigma=0$ one can operate on the coefficients of polynomials of growing degree. Furthermore, due to the lower triangular structure of the matrix at the left side in \eqref{eq:update_formula_approach1}, the main computational cost of the extension step consists of solving a system in $M(\jmath\omega)$ which is of dimension $2n$ and whose factorisation needs to be computed only once. However, there are three difficulties that render this approach unsuited in practice.
\begin{enumerate}
	\item Firstly, although the functions generated by \Cref{alg:structure_preserving_shift_invert_arnoldi} are uniquely defined by the starting function $\phi_1$ and can be uniquely decomposed as in \eqref{eq:preserverd_structure} when such a structure is imposed on $\phi_1$, the decomposition in terms of $e^{\jmath\omega \theta}$ and $e^{-\jmath \omega \theta}$ is not uniquely defined on $X$. For example, let $f \in X$ and consider the decomposition
	$
	f(\theta) = P(\theta) e^{\jmath \omega\theta} + \overline{P(\theta)} e^{-\jmath\omega \theta},
	$
	then
	$
	R(\theta) e^{\jmath \omega \theta} + \overline{R(\theta)} e^{-\jmath \omega \theta}
	$
	with $R(\theta) = P(\theta)+\sin(\omega \theta) +\jmath \cos(\omega \theta)$ is another decomposition for $f$. This non-uniqueness causes problems when working in finite precision. For example, the range of the polynomial $P_{N_i}(\cdot)$ can grow large, while the range of the overall function $\phi_i$ remains small. Another related numerical issue that we observed, the condition number of the matrix at the left-hand side of \eqref{eq:update_formula_approach1} grows large as the number of iterations increases.
	
	
	\item Secondly, evaluating inner product \eqref{eq:inner_product} is now less trivial. An alternative would be: rather than using the coefficients of the Chebyshev expansion, use the coefficients of the polynomial (expressed in the scaled Chebyshev basis) that interpolates the function in a  number of scaled Chebyshev points. If the number of interpolation points is sufficiently large, then the coefficients of this interpolating polynomial are a good approximation for the coefficients in the Chebyshev expansion  \cite{Trefethen2013}. The coefficients of this interpolating polynomial can be obtained using the \texttt{(i)fft}-transformation at a cost of $\mathcal{O}\Big(2n\, N_{points} \log(N_{points}) \Big)$ with $N_{points}$ the number of discretisation points \cite{Ahmed1968}.
	
	\item Finally, when using the representation \eqref{eq:preserverd_structure}, the simultaneous orthogonalisation with respect to the inner product \eqref{eq:inner_product} and bilinear form $\bilinear{\cdot}{J\cdot}$ typically destroys the structure of $\phi_{i+1}$. 
\end{enumerate}
The three difficulties mentioned above render the representation used in \Cref{theorem:non_zero_shift} unattractive in practice. Next, we therefore introduce a different approach to implement \Cref{alg:structure_preserving_shift_invert_arnoldi} for $\sigma\neq 0$ using finite-dimensional operations. Similarly as in Chebfun \cite{Trefethen2013}, we approximate the functions in $X$ using polynomial approximations whose degree is adaptively chosen such that they match the original functions up to machine precision. More specifically, a function $\psi$ can be approximated by the unique interpolating polynomial $\widehat{\psi}$ of degree $N_d$ that interpolates $\psi$ in the extreme points of the $N_d$\textsuperscript{th} Chebyshev polynomial of the first kind  rescaled to the interval $[-\tau_K,\tau_K]$: 
\begin{equation}
\label{eq:cheb_approximation_form}
\textstyle
\widehat{\psi}(\theta) = \sum\limits_{l=0}^{N_d} c_{l} T_l\Big(\frac{\theta}{\tau_K}\Big) \text{ such that } \widehat{\psi}(\theta_l) = \psi(\theta_l) \text{ with } \theta_l = \tau_K\cos\big(\frac{l\pi}{N_d}\big) \text{ for } l=0,\dots,N_d,
\end{equation}
with $N_d$ chosen such that the approximation error is sufficiently small.  As mentioned before, such an interpolating polynomial can be computed efficiently using the \texttt{(i)fft}-transformation.

Previously, by invoking \Cref{theorem:zero_shift,theorem:non_zero_shift}, $\mappedoperator^{-1}\phi_i$ could be computed efficiently due to the chosen structure of $\phi_1$ (which was preserved in $\phi_i$). More specifically, for large $n$, the main computation cost of the extension step consisted of solving a system involving a particular instance of the characteristic matrix \eqref{eq:characteristic_matrix}. However, for arbitrary (non-structured) functions $\phi_i$ computing $\mappedoperator^{-1}\phi_i$ involves solving an ordinary differential equation subjected to conditions \eqref{eq:structure_preserving_operator_D1}-\eqref{eq:structure_preserving_operator_D2}. A finite-dimensional approximate solution for this differential equation can be obtained using spectral discretisation. This would however involve solving a system of dimension $2nN_d$ with $N_d$ a sufficiently large number of discretisation points. To avoid having to solve a system with such dimensions, we will use the approach depicted in \Cref{fig:structure_approach_non_zero_shift}, which is based on explicitly using the  facotriazation $\mappedoperator^{-1}=(\operator+\sigma \mathcal{I}_X)^{-1}  (\operator-\sigma \mathcal{I}_X)^{-1}$. First, $\phi_i(\cdot)$ is approximated by the function $\chi_i(\cdot) e^{\sigma\cdot}$ in which $\chi_i(\cdot)$ is the interpolating polynomial of form \eqref{eq:cheb_approximation_form} of $\phi_i(\cdot)e^{-\sigma \cdot}$. Next one can solve $(\operator-\sigma \mathcal{I}_{X})^{-1}\chi_i(\cdot) e^{\sigma\cdot}$ analytically. More specifically, $(\operator-\sigma \mathcal{I}_{X})^{-1}\chi_i(\cdot) e^{\sigma\cdot}$ is equal to $\xi_i(\cdot) e^{\sigma \cdot}$ with $\xi_i(\cdot)$ a vector-valued polynomial whose degree is equal to one plus the degree of $\chi_i(\cdot)$. Using the notation
\[
\chi_{i}(\theta) = \sum\limits_{l=0}^{N_1} a_l T_l\left(\frac{\theta}{\tau_K}\right) \text{ and } \xi_{i}(\theta) = \sum\limits_{l=0}^{N_{1}+1} b_l T_l\left(\frac{\theta}{\tau_K}\right)
\]
the coefficients of $\xi_i(\cdot)$ are given by
\begin{equation}
\label{eq:update_non_zero_1a}
\begin{bmatrix}
b_1 & \dots & b_{N_{1}+1}\end{bmatrix} = \begin{bmatrix}
a_0 & \dots & a_{N_{1}}\end{bmatrix}  \tau_K \begin{bmatrix}
1 & &  & & \\
& \frac{1}{4} & & & \\
\frac{-1}{2} & & \frac{1}{6} & & \\
& \ddots & & \ddots & \\
& \multicolumn{2}{c}{\frac{-1}{2(N_{1}-1)}} & \multicolumn{2}{l}{\frac{1}{2(N_{1}+1)}}
\end{bmatrix}
\end{equation}
and 	
\begin{multline}
\label{eq:update_non_zero_1b}
M(\sigma)\ b_0= -\chi_i(0)+(H_0-\sigma I_n) \sum_{l=1}^{N_1+1} b_l T_l(0) + \\ \sum_{k=1}^{K}\left(H_k e^{\sigma \tau_k} \sum_{l=1}^{N_1+1} b_l T_l\left(\frac{\tau_k}{\tau_K}\right)+ H_{-k} e^{-\sigma \tau_k}\sum_{l=1}^{N_1+1} b_l T_l\left(-\frac{\tau_k}{\tau_K}\right)\right).
\end{multline}
with $M(\cdot)$ the characteristic matrix as defined in \eqref{eq:characteristic_matrix}. As before, one now only has to solve a system of dimension $2n$. Notice that these expressions are similar to those in the extension step of the standard infinite Arnoldi method from \cite{Jarlebring2010}. Subsequently, we approximate $\xi_{i}(\cdot)e^{\sigma \cdot}$ by $\zeta_i(\cdot) e^{-\sigma \cdot}$ with  $\zeta_i(\cdot)$ the interpolating polynomial of form \eqref{eq:cheb_approximation_form} of $ \xi_i(\cdot)e^{2\sigma\cdot}$. In step \textbf{(IV)} we compute $\left(\mathcal{H}+\sigma \mathcal{I}_X\right)^{-1}\zeta_i(\cdot) e^{-\sigma \cdot}$ for which we can again derive an analytical expression. More specifically, $\left(\mathcal{H}+\sigma \mathcal{I}_X\right)^{-1}\zeta_i(\cdot) e^{-\sigma \cdot}$ is equal to $\Upsilon_{i}(\cdot) e^{-\sigma \cdot}$ with $\Upsilon_{i}(\theta) =  \sum\limits_{l=0}^{N_{2}+1} d_l T_l\left(\frac{\theta}{\tau_K}\right)$ a vector-valued polynomial whose coefficients follow from those of 
$\zeta_{i}(\theta) = \sum\limits_{l=0}^{N_{2}} c_l T_l\left(\frac{\theta}{\tau_K}\right)$ by the following relations
\begin{equation}
\label{eq:update_non_zero_2a}
\begin{bmatrix}
d_1 & \dots & d_{N_{2}+1}\end{bmatrix} = \begin{bmatrix}
c_0 & \dots & c_{N_{2}}\end{bmatrix} \tau_K \begin{bmatrix}
1 & &  & & \\
& \frac{1}{4} & & & \\
\frac{-1}{2} & & \frac{1}{6} & & \\
& \ddots & & \ddots & \\
& \multicolumn{2}{c}{\frac{-1}{2(N_{d,2}-1)}} & \multicolumn{2}{l}{\frac{1}{2(N_{d,2}+1)}}
\end{bmatrix}
\end{equation}
and
\begin{multline}
\label{eq:update_non_zero_2b}
M(-\sigma)\ d_0=-\zeta_i(0)+(H_0+\sigma I_n) \sum_{l=1}^{N_{2}+1} d_l T_l(0) + \\ \sum_{k=1}^{K}\left( H_k e^{-\sigma \tau_k} \sum_{l=1}^{N_{2}+1} d_l T_l\left(\frac{\tau_k}{\tau_K}\right)+ H_{-k} e^{\sigma \tau_k}\sum_{l=1}^{N_{2}+1} d_l T_l\left(-\frac{\tau_k}{\tau_K}\right)\right).
\end{multline}
Finally, $\widehat{\varphi}_{i+1}$ is obtained by computing the interpolating polynomial of form \eqref{eq:cheb_approximation_form} of $\Upsilon_{i}(\cdot)e^{-\sigma \cdot}$. If the interpolating polynomials in steps \textbf{(I)}, \textbf{(III)} and \textbf{(V)} are computed up to machine precision, the resulting $\widehat{\varphi}_{i+1}$ is an accurate approximation for $\varphi_{i+1}=\mappedoperator^{-1}\phi_i$.

\begin{figure}[!h]
	\centering
	\begin{tikzpicture}[scale=0.88, every node/.style={scale=0.88}]
	\node[] at (-3.75,4) (step1) {\hspace{-2cm}$\phi_i(\theta) = \sum\limits_{l=0}^{N_{\phi_i}} q_l^{(i)} T_l\left(\frac{\theta}{\tau_K}\right) $};
	\node[] at (-3.75,0) (step2) {$\chi_{i}(\theta) = \sum\limits_{l=0}^{N_1} a_l T_l\left(\frac{\theta}{\tau_K}\right)$};
	\node[] at (-3.75,-3.75) (step3) {\hspace*{-5em}$\xi_{i}(\theta) = \sum\limits_{l=0}^{N_{1}+1} b_l T_l\left(\frac{\theta}{\tau_K}\right)$};
	\node[] at (3.75,-3.75) (step4) {$\zeta_{i}(\theta) = \sum\limits_{l=0}^{N_{2}} c_l T_l\left(\frac{\theta}{\tau_K}\right)$\hspace*{-6em}};
	\node[] at (3.75,0) (step5) {$\Upsilon_{i}(\theta) = \sum\limits_{l=0}^{N_{2}+1} d_l T_l\left(\frac{\theta}{\tau_K}\right)$\hspace{6em}};
	\node[] at (3.75,4) (step6) {$\widehat{\varphi}_{i+1}(\theta) = \sum\limits_{l=0}^{N_{\varphi_{i+1}}} v_l^{(i+1)} T_l\left(\frac{\theta}{\tau_K}\right)$\hspace*{-2cm}};
	\node[draw,dashed,very thick,black!30,label=above:{Representation as polynomial},fit={(step1.north west) (step6.south east)},rounded corners=0.3cm,inner xsep=26mm,inner ysep = 3mm]  (box1) {};
	\node[draw,dashed,very thick,black!30,label=below:{Representation as polynomial times $e^{\sigma \theta}$},fit={(step2.north west) (step3.south east-|step2.east)},rounded corners=0.3cm,inner xsep=17.5mm,inner ysep = 3mm]  (box2) {};
	\node[draw,dashed,very thick,black!30,label=below:{Representation as polynomial times $e^{-\sigma \theta}$},fit={(step5.north west) (step4.south east-|step5.east)},rounded corners=0.3cm,inner xsep=19mm,inner ysep = 3mm]  (box3) {};
	\draw[-stealth,line width=1pt,black] (step1) -- (step2) node[midway,text width=5cm,fill=white,align=center]{\textbf{(I)} Compute the interpolating polynomial $\chi_i(\cdot)$ of $\phi_i(\cdot)e^{-\sigma \cdot}$ s.t. $\phi_i(\cdot) \approx \chi_i(\cdot) e^{\sigma \cdot}$};
	\draw[-stealth,line width=1pt] (step2) -- (step3.north-|step2) node[midway,fill=white,align=center] {\textbf{(II)} Solve $\left(\operator - \sigma \mathcal{I}_{X}\right)\left(\xi_i(\cdot)e^{\sigma \cdot}\right) = \chi_i(\cdot)e^{\sigma \cdot} $ \\ using \eqref{eq:update_non_zero_1a}-\eqref{eq:update_non_zero_1b} s.t. \\ $\xi_i(\cdot)e^{\sigma \cdot} \approx \left(\operator - \sigma \mathcal{I}_{X}\right)^{-1} \phi_i(\cdot) $};
	\draw[-stealth,line width=1pt] (step3) -- (step4) node[midway,fill=white,align=center] {\textbf{(III)} Compute the interpolating\\ polynomial $\zeta_i(\cdot)$ of $\xi_{i}(\cdot)e^{2\sigma\cdot}$ \\ s.t.  $\zeta_i(\cdot)e^{-\sigma \cdot} \approx \xi_{i}(\cdot) e^{\sigma \cdot}$};
	\draw[-stealth,line width=1pt] (step4.north-| step5) -- (step5) node[midway,fill=white,align=center] {\textbf{(IV)} Solve $\left(\operator + \sigma \mathcal{I}_{X}\right)\left(\Upsilon_i(\cdot)e^{\negative\sigma \cdot}\right)=\zeta_i(\cdot)e^{\negative\sigma \cdot} $\\using \eqref{eq:update_non_zero_2a}-\eqref{eq:update_non_zero_2b} s.t. \\ $\begin{aligned}\Upsilon_i(\cdot)e^{-\sigma \cdot} \approx& (\operator+\mathcal{I}_X\sigma)^{-1}\left(\xi_i(\cdot)e^{\sigma \cdot}\right) \\ \approx&  (\operator+\mathcal{I}_X\sigma)^{-1}(\operator-\mathcal{I}_X\sigma)^{-1}\phi_i(\cdot)\end{aligned}$};
	\draw[-stealth,line width=1pt] (step5) -- (step6) node[midway,fill=white,text width=6cm,align=center] {\textbf{(V)} Compute the interpolating polynomial $\widehat{\varphi}_{i+1}(\cdot)$ of $\Upsilon_i(\cdot)e^{-\sigma \cdot}$ s.t. \\ $\widehat{\varphi}_{i+1} \approx (\operator+\mathcal{I}_X\sigma)^{-1}(\operator-\mathcal{I}_X\sigma)^{-1}\phi_i$};
	\draw[-stealth,line width=1pt] (step6) -- (step1) node[midway,align=center,fill=white] {\textbf{Orthogonalise} \\ 
		wrt \eqref{eq:inner_product} and $\bilinear{\cdot}{J\cdot}$ \\ $i\leftarrow i+1$};
	
	\end{tikzpicture}
	\caption{Structured representation of the used numerical implementation of \Cref{alg:structure_preserving_shift_invert_arnoldi} for $\sigma \neq 0$.}
	\label{fig:structure_approach_non_zero_shift}
\end{figure}
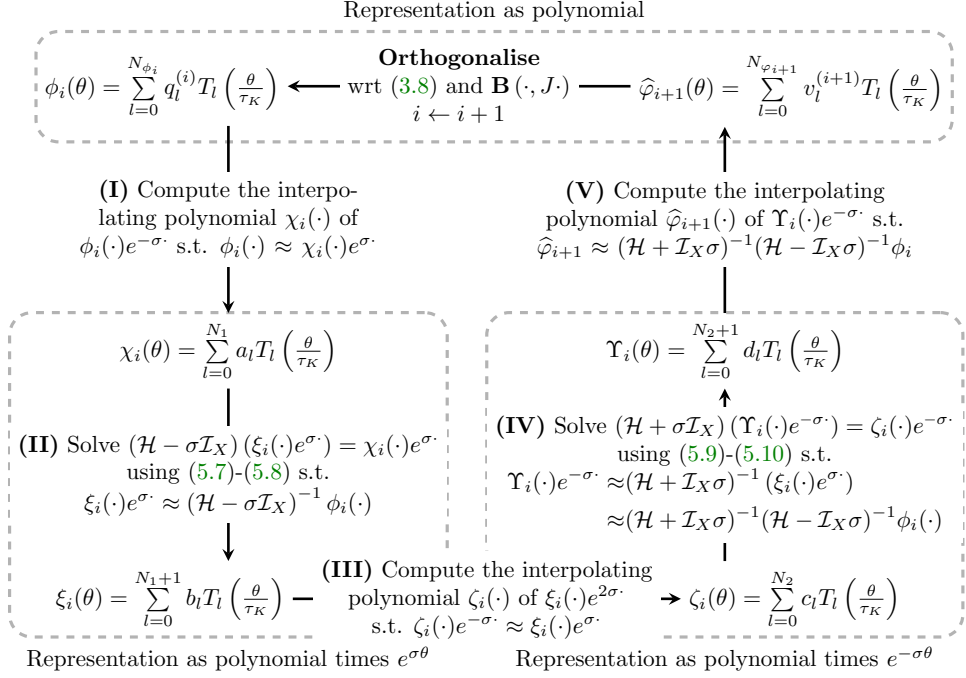

Using this polynomial representation expressed in the Chebyshev basis, also the orthogonalisation step can be implemented with finite-dimensional operations. More specifically, by using the representation of \Cref{fig:structure_approach_non_zero_shift} for  $\phi_1,\dots,\phi_{i}$ and $\hat{\varphi}_{i+1}$, we can use an  orthogonalisation procedure similar to that in \eqref{eq:orthogonalisation_zero} to ensure orthogonality with respect to both the inner product \eqref{eq:inner_product} and bilinear form $\bilinear{\cdot}{J\cdot}$: instead of padding the vectors to length $2n(N_{i+1}+1)$, the vectors now must be padded to length $2n\Big(\max\{ \max_{j=1,..,i} \{N_{\phi_j}\}, N_{\hat{\varphi}_{i+1}}\}+1\Big)$.

A disadvantage of the approach introduced above is that, in contrast to \Cref{theorem:zero_shift,theorem:non_zero_shift}, the degree of the polynomial $\hat{\varphi}_{i+1}$ is not known beforehand and therefore might grow large. However, the example in \Cref{subsec:example2} shows that if $\sigma$ is sufficiently small, the degree of these functions only grows slowly. For large $\sigma$ this is often no longer the case as the sought for eigenfunctions are either highly oscillator (large imaginary component) or fast growing exponentials (large real component). A polynomial representation inherently requires a high degree to accurately approximate such functions. 
\section{Numerical illustration}
\label{sec:numerical_illustration}
In this section we consider two examples. The first example illustrates the importance of the chosen shift-invert transformation and of the explicit orthogonalisation of basis functions of the Krylov subspace with respect to $\bilinear{\cdot}{J\cdot}$ for computing purely imaginary eigenvalues of \eqref{eq:nlevp} in the presence of rounding errors. The second example shows the convergence behavior of the presented method for a large-scale example. oth experiments are carried out in \textsc{Matlab} R2022a (version 9.12.0.1884302) on a Dell Latitude 7490 with a Intel(R) Core(TM) i7-8650U CPU @ 1.90GHz eight core processor with 16GB of RAM running Ubuntu 20.04.4 LTS. The code for these experiments is available from \url{https://twr.cs.kuleuven.be/research/software/delay-control/SPSIIA/index.html}.
\subsection{Example 1}
\label{subsec:example1}
Consider the following characteristic matrix of form \eqref{eq:characteristic_matrix}:
\begin{equation}
\label{eq:example1}
\lambda
\begin{bmatrix}
1 & 0 \\
0 & 1
\end{bmatrix}
- \begin{bmatrix}
10 & 0.1 \\
c_0  & -10
\end{bmatrix} -
\begin{bmatrix}
a_1 & 0 \\
0 & 0
\end{bmatrix} e^{-\lambda}
- \begin{bmatrix}
0 &  0 \\
0 & -a_1
\end{bmatrix} e^{\lambda},
\end{equation}
with $a_1 = (3\pi^2/4)/(20+\pi)$ and $c_0 = -1000-10a_1^2-10a_1\pi-\frac{5\pi^2}{2}$,
which has purely imaginary eigenvalues at $\jmath\pi$ and $\frac{\jmath\pi}{2}$. \Cref{table:example1,table:example1b} compare three methods for computing the purely imaginary eigenvalues of the associated NLEVP using $\sigma=0$ and $\sigma = \jmath\frac{3\pi}{4}$, respectively. These three methods are:
\begin{enumerate}
	\item applying the infinite Arnoldi method to $\operator^{-1}$ and $\left(\operator-\sigma \mathcal{I}_{X}\right)^{-1}$, respectively, using a modification of the algorithm from \cite{Jarlebring2010} to deal with both positive and negative delays;
	\item applying the infinite Arnoldi method to $\mathcal{R}_0^{-1}$ and $\mappedoperator^{-1}$ using the method presented in \Cref{sec:shift_zero} and \Cref{sec:non_zero_shift} (\Cref{fig:structure_approach_non_zero_shift}), respectively, but without explicit orthogonalisation with respect to $\bilinear{\cdot}{J\cdot}$ (\ie{} using \eqref{eq:orthogonalisation_naive} instead of \eqref{eq:orthogonalisation_zero});
	\item applying the infinite Arnoldi method to $\mathcal{R}_0^{-1}$ and $\mappedoperator^{-1}$ using the method presented in \Cref{sec:shift_zero} and \Cref{sec:non_zero_shift} (\Cref{fig:structure_approach_non_zero_shift}), respectively, with explicit orthogonalisation with respect to $\bilinear{\cdot}{J\cdot}$.
\end{enumerate} As initial function the constant function $\begin{bmatrix}
0.6 & 0.8
\end{bmatrix}^{\top}$ is used in all cases but one. When applying the infinite Arnoldi method to $\left(\operator-\sigma \mathcal{I}_{X}\right)^{-1}$, the function $\begin{bmatrix}
0.6 & 0.8
\end{bmatrix}^{\top}e^{\sigma \theta}$ is employed, to be able to use the compact representation from \cite{Jarlebring2010} for the functions in the Krylov subspace.

For the first approach, we observe from \Cref{table:example1,table:example1b} that the obtained approximations have a significant non-zero real part, as this method does not explicitly take into account the symmetry of the spectrum with respect to the imaginary axis. For the second approach, we see that each purely imaginary eigenvalue of the NLEVP appears twice in the table. Recall that the eigenvalues $\pm\jmath\omega$ of $\operator$ are mapped to the real eigenvalue $\mu = \frac{1}{-\omega^2-\sigma^2}$ of $\mappedoperator^{-1}$ with multiplicity 2. The eigenspace associated with $\mu$ is given by $\mathcal{L}=\Span\{\varphi_{+},\varphi_{-}\}$, with $\varphi_{+}$ and $\varphi_{-}$ the eigenfunctions of $\mathcal{H}$ associated with $\jmath\omega$ and $-\jmath\omega$, respectively. Although in exact arithmetic we expect that only one linear independent component of $\mathcal{L}$ will be approximated by the Krylov subspace (a consequence of \Cref{theorem:uniqueness}), this is typically no longer the case when working in finite precision as orthogonality with respect to $\bilinear{\cdot}{J\cdot}$ is lost due to rounding error. As a consequence, the Krylov subspace will eventually also approximate a second linear independent component in the eigenspace $\mathcal{L}$, meaning the double eigenvalue $\mu$ of $\mappedoperator^{-1}$ will be approximated by a pair of Ritz values. One either gets two nearby real Ritz values or a pair of complex conjugate Ritz values. In the latter case, after re-transformation the obtained approximations for purely imaginary eigenvalues of $\mathcal{H}$ thus have a small real part. When, however, explicit orthogonalisation against $\bilinear{\cdot}{J\cdot}$ is preformed (approach 3), the Krylov subspace satisfies the condition in \Cref{theorem:J_neutrality} up to machine precision. As to be expected, the approximations for the purely imaginary eigenvalues are now purely imaginary.

Finally, from \Cref{table:example1b}, we observe that when using the shift-invert transformation $\left(\operator-\sigma \mathcal{I}_{X}\right)^{-1}$, the eigenvalues $\jmath\omega$ and $-\jmath\omega$ are approximated separately. In contrast, when using $\mappedoperator^{-1}$ this pair is really approximated as a pair.
\begin{table}[!h]
	\caption{Obtained approximations after 21 iterations for the purely imaginary eigenvalues of \eqref{eq:example1} using approaches 1-3 for $\sigma = 0$ with the constant function $\begin{bmatrix}
		0.6 & 0.8
		\end{bmatrix}^{\top}$ as initial function.}
	\label{table:example1}
	\centering
	\resizebox{\textwidth}{!}{
		\begin{tabular}{c|c|c}
			Approach 1 & Approach 2 & Approach 3   \\
			\hline & & \\[-1.5ex]
			$-2.525\times 10^{-11} - \jmath\underline{1.570796326}838293$ & $\phantom{-}1.444\times 10^{-12}-\jmath \underline{1.570796326794}669$ &   $\phantom{-}\jmath\underline{1.5707963267}50096$\\
			$-2.525\times10^{-11} + \jmath\underline{1.570796326}838293$ & $\phantom{-}1.444\times 10^{-12}+\jmath \underline{1.570796326794}669$ & $-\jmath\underline{1.5707963267}50096$\\
			$-5.141\times10^{-9} - \jmath \underline{3.141592653}006948$& $-1.444\times 10^{-12}-\jmath \underline{1.570796326794}669$ &  $\phantom{-}\jmath \underline{3.141592653}831962$\\
			$-5.141\times10^{-9} + \jmath \underline{3.141592653}006948$& $-1.444\times 10^{-12}+\jmath \underline{1.570796326794}669$ &$-\jmath\underline{3.141592653}831962$ \\
			& $-\jmath \underline{3.1415926535}91730$ & \\
			& $\phantom{-}\jmath \underline{3.1415926535}91730$ & \\
			& $-\jmath \underline{3.141592653}286507$& \\
			&$\phantom{-}\jmath \underline{3.141592653}286507$ & 
	\end{tabular}}
\end{table}
\begin{table}[!h]
	\caption{Obtained approximations after 21 iterations for the purely imaginary eigenvalues of \eqref{eq:example1} using approaches 1-3 for $\sigma = \jmath\frac{3\pi}{4}$. For the approach based on $\left(\operator-\sigma I_X\right)^{-1}$ the initial function $\begin{bmatrix}
		0.6 & 0.8
		\end{bmatrix}^{\top}e^{\sigma \theta}$ is used, while for the approaches based on $\mappedoperator^{-1}$ the initial function $\begin{bmatrix}
		0.6 & 0.8
		\end{bmatrix}^{\top}$ is used.}
	\label{table:example1b}
	\centering
	\resizebox{\textwidth}{!}{
		\begin{tabular}{c|c|c}
			Approach 1 & Approach 2 & Approach 3  \\ 
			\hline & & \\[-1.5ex]
			$\phantom{-}3.005\times 10^{-13} + \jmath \underline{1.57079632679}2780$ & $8.260\times 10^{-13} - \jmath \underline{1.570796326794}206$&  $-\jmath \underline{1.570796326}657866$ \\
			$-3.238\times 10^{-13} + \jmath \underline{3.1415926535}90773$ & $8.260\times 10^{-13} + \jmath \underline{1.570796326794}206$ & $\phantom{-}\jmath \underline{1.570796326}657866$\\
			$\phantom{-}1.570\times 10^{-6} - \jmath \underline{1.57070}8037662637$ & $-8.260\times 10^{-13} - \jmath \underline{1.570796326794}206$ & $-\jmath \underline{3.141592653}619107$ \\
			$-5.788\times 10^{-4} - \jmath \underline{3.14}0992633763653$ & $-8.260\times 10^{-12} + \jmath \underline{1.570796326794}206$ & $\phantom{-}\jmath \underline{3.141592653}619107$\\
			& $-\jmath \underline{3.1415926535}91559$ & \\
			& $\phantom{-}\jmath \underline{3.1415926535}91559$ & \\
			& $-\jmath \underline{3.141592653589}887$ &  \\
			& $\phantom{-}\jmath \underline{3.141592653589}887$ & 
	\end{tabular}}
\end{table}
\subsection{Example 2}
\label{subsec:example2}
For the second example we consider the following dynamical system from \cite{Michiels2008}, which describes a heated rod which is cooled using delayed feedback. The evolution of the temperature in the rod, $v$, is governed by the partial differential equation
\[
\frac{\partial v(x,t)}{\partial t} = \frac{\partial^2 v(x,t)}{\partial x^2} + a_0(x) v(x,t) + a_1(x)v(\pi-x,t-1) \text{ for } x\in[0,\pi],
\]
with $a_0(x) = - 2\sin(x)$, $a_1(x) = 2\sin(x)$ and $v(0,t)=v(\pi,t) = 0$. Discretising this partial differential equation in the space coordinate $x$ results in a system of delay-differential equations of dimension $n$. To obtain a dynamical system of form \eqref{eq:dyn_sys}, we define a performance output matrix $C=\frac{1}{n}\begin{bmatrix}
1 & \dots & 1
\end{bmatrix}$ (which gives the average temperature of the rod) and a performance input matrix $B = C^{\top}$. Choosing $n=1000$ and plotting the singular value of the transfer matrix in function of $\omega$, one observes that this singular value is equal to $0.00018$ for $\omega$ approximately equal to $2.009437$, $3.790888$ and $5.571120$.  It now follows from the discussion in the introduction that for $\gamma = 0.00018$, the NLEVP associated with \eqref{eq:nlevp_hinf} must have purely imaginary eigenvalues around $\jmath 2.009437$, $\jmath 3.790888$ and $\jmath 5.571120$.

To verify this, we first use the method presented in \Cref{sec:shift_zero} to compute the eigenvalues of \eqref{eq:nlevp_hinf} near the shift $\sigma = 0$. \Cref{fig:example2_omega0} shows the obtained approximations for the eigenvalues in the region $[-6,6] \times \jmath[-10,10]$ after 70 iterations and their convergence behavior. One observes that the approximations converge quickly to the eigenvalues of \eqref{eq:nlevp_hinf} close to the origin.

Next, we apply the method from \Cref{sec:non_zero_shift}, with implementation sketched in \Cref{fig:structure_approach_non_zero_shift}, for $\sigma = \jmath 4.5$. \Cref{fig:example2_omega45} shows the eigenvalues and the obtained approximations near this shift after 70 iterations. We have again fast convergence to the eigenvalues near the shift. 

Finally, recall that the degree of the polynomial $\hat{\varphi}_{i+1}$, defined in \Cref{fig:structure_approach_non_zero_shift}, is not known beforehand. \Cref{fig:example2_evolution_degree} therefore shows the evolution of the degree of this polynomial with respect to the iteration number. We observe that the degree of these polynomials grows slowly with respect to the number of iterations.

\begin{figure}[!h]
	\subfloat[Eigenvalues (crosses) and their approximations (circles) near $\sigma = 0$.]{
		\includegraphics[width=0.48\linewidth]{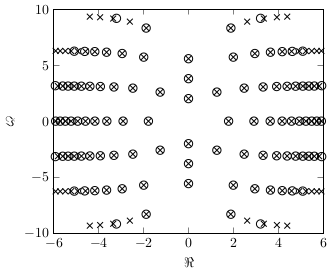}	
	}
	\subfloat[Relative forward error]{
		\includegraphics[width=0.48\linewidth]{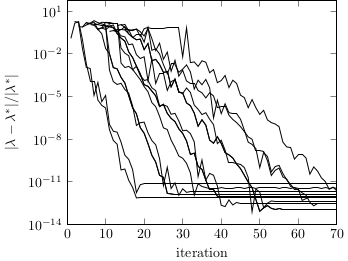}
	}
	\caption{Approximated eigenvalues near the shift $\sigma = 0$ after 70 iterations and their convergence behavior for Example~2 obtained using the method from \Cref{sec:shift_zero}.}
	\label{fig:example2_omega0}
\end{figure}
\begin{figure}[!h]
	\subfloat[Eigenvalues (crosses) and their approximations (circles) near $\sigma = \jmath4.5$.]{
		\includegraphics[width=0.48\linewidth]{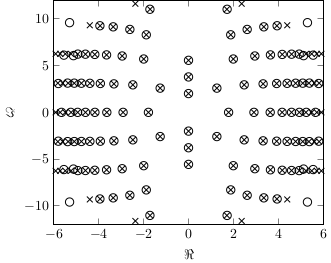}
	}
	\subfloat[Relative forward error]{
		\includegraphics[width=0.48\linewidth]{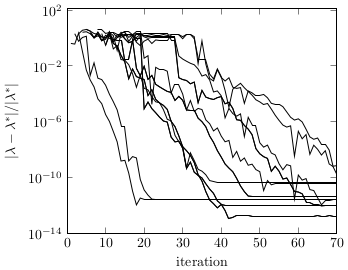}
	}
	\caption{Approximated eigenvalues near the shift $\sigma = \jmath4.5$ after 70 iterations and their convergence behavior for Example~2 obtained using the method from \Cref{sec:non_zero_shift}, as depicted in \Cref{fig:structure_approach_non_zero_shift}.}
	\label{fig:example2_omega45}
\end{figure}
\begin{figure}[!h]
	\centering
	\begin{tikzpicture}
	\begin{axis}[xlabel = $i$,
	ylabel = degree$(\widehat{\varphi}_{i+1})$,
	xmin=1,
	xmax=71,
	ymin=0,
	width=0.6\textwidth
	]
	\addplot[mark=x,draw=black]%
	table[x index=0,y index=1]{Results/example2bdegree.txt} ;
	\end{axis}
	\end{tikzpicture}
	\caption{Evolution of the degree of $\widehat{\varphi}_{i+1}$, as defined in \Cref{fig:structure_approach_non_zero_shift}, as function of the iteration number~$i$.}
	\label{fig:example2_evolution_degree}
\end{figure}
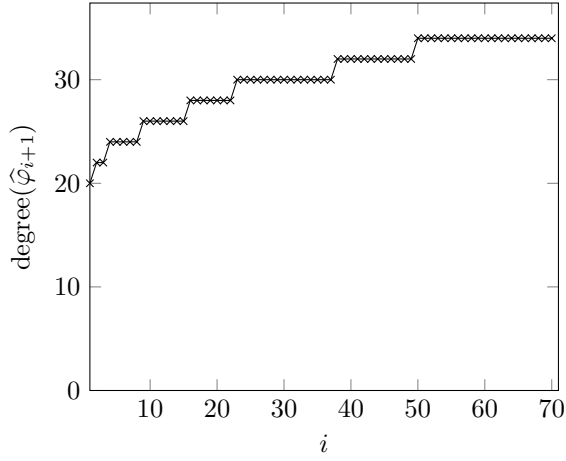
\section{Conclusions and outlook}
\label{sec:conclusions}
In this work we presented an iterative method to approximate the eigenvalues of NLEVPs of form \eqref{eq:nlevp} closest to a given shift $\sigma$ while preserving the symmetries of the spectrum. The presented work can thus be seen as a generalization of the results from \cite{Mehrmann2001,Mehrmann2012} to a class of NLEVPs. 

To derive this method, the equivalence between the considered NLEVP and a linear but infinite-dimensional eigenvalue problem was used. Based on this equivalence, we introduced a shift-invert transformation that preserves the Hamiltonian structure of the spectrum. Next, the ideas behind the infinite Arnoldi method from \cite{Jarlebring2010}, which operates on functions rather than on vectors, were applied to this transformed eigenvalue problem to construct a Krylov subspace. It was then shown that this subspace is orthogonal with respect to the bilinear functional $\bilinear{\cdot}{J\cdot}$. This result was subsequently used to demonstrate that simple purely imaginary eigenvalues of \eqref{eq:nlevp} close to the shift are generally approximated by purely imaginary eigenvalues. Although this method was initially defined on function spaces, \Cref{sec:shift_zero,sec:non_zero_shift} showed how it can be implemented using finite-dimensional linear algebra operations. The performance of these numerical algorithms was finally verified using two numerical experiments in \Cref{sec:numerical_illustration}.

To conclude this paper we give some directions for future research.	Firstly, as mentioned in the introduction, the presented algorithm can be used as a building block for algorithms that compute the \hinfnorm{} of time-delay systems. Secondly, a more extensive study on the effect of the chosen initial function and the chosen inner product on the convergence behavior of the method is necessary. Thirdly, for large $\sigma$ the representation in \eqref{eq:preserverd_structure} has as advantage that it might require polynomials of a lower degree to approximate the eigenfunctions of nearby eigenvalues in comparison to purely polynomial approximations. The functions $e^{\jmath \omega \theta}$ and $e^{-\jmath\omega \theta}$ act in this case as carrier functions. However, as mentioned before, it is not yet clear how this representation can be used in practice because of problems with numerical stability. Finally, for the infinite Arnoldi method in \cite{Jarlebring2010}, the convergence behavior of the method can be related to the approximation error of a Padé-approximation of the DEVP with growing degree on the one hand and the convergence behavior of the finite-dimensional Arnoldi method on the other hand. Such a connection is yet to be established for the method presented here.

\appendix

\section{Proof that $\bilinear{\cdot}{J\cdot}$ is anti-symmetric}
\label{appendix:anti_symmetry_bilinear_J_form}
\begin{lemma}
	The bilinear form $\bilinear{\cdot}{J\cdot}$ is anti-symmetric, this means that the equality $\bilinear{\varphi}{J\psi} = - \bilinear{\psi}{J\varphi}$ holds.
\end{lemma}
\begin{proof}
{\footnotesize
\begin{align*}
\bilinear{\varphi}{J\psi} &= \textstyle - \Big[\psi(0)^{\top}J\varphi(0) + \sum\limits_{k=1}^{K} \Big(\int\limits_{0}^{\tau_k} \psi(\theta)^{\top}JH_{\negative k} \varphi(\theta-\tau_k) \dd\theta - \int\limits_0^{\tau_k} \psi(\theta-\tau_k)^{\top} JH_k \varphi(\theta) \dd\theta \Big) \Big] \\
&= \textstyle - \Big[\varphi(0)^{\top}J^{\top}\psi(0) +\! \sum\limits_{k=1}^{K}\! \Big(\int\limits_{0}^{\tau_k} \varphi(\theta-\tau_k)^{\top}\!(JH_{\negative k})^{\top}\! \psi(\theta) \dd\theta -\! \int\limits_0^{\tau_k}  \varphi(\theta)^{\top}\!(JH_k)^{\top}\! \psi(\theta-\tau_k) \dd\theta \Big)\Big] \\
&= \textstyle \Big[\varphi(0)^{\top}\!J\!\psi(0) + \sum\limits_{k=1}^{K} \Big( \int\limits_0^{\tau_k}  \varphi(\theta)^{\top} JH_{\negative k} \psi(\theta-\tau_k) \dd\theta -  \int\limits_{0}^{\tau_k} \varphi(\theta-\tau_k)^{\top} JH_{k} \psi(\theta) \dd\theta \Big)\Big] \\
&= - \bilinear{ \psi}{ J\varphi}
\end{align*}}%
\end{proof}

\section{Derivation of the extension step for $\sigma = 0$}
\label{appendix:extension_step_sigma0}
Let $\phi_i(\theta)$ and $\varphi_{i+1}(\theta)$ be as defined in \Cref{theorem:zero_shift}, then the equality $\mathcal{R}_0 \varphi_{i+1} = \phi_{i}$ becomes
\[
\Big(\mathcal{R}_0 \varphi_{i+1}(\theta) =\Big) \sum_{l=2}^{N_i+2} \frac{v_l^{(i+1)}}{(\tau_K)^2} T_l''\left(\frac{\theta}{\tau_K}\right) = \sum_{l=0}^{N_i} q_{l}^{(i)} T_l\left(\frac{\theta}{\tau_K}\right) \Big(=\phi_i(\theta) \Big).
\]
Expression \eqref{eq:extension_sigma0} now follows by noting that 
\begin{equation}
\label{eq:chebyshev_derivative2}
\begin{split}
T_0(t) =& T_2''(t)/4, \\
T_1(t) =& T_3''(t)/24, \\
T_2(t) =& T_4''(t)/48-T_2''(t)/6 \text{ and }\\
T_l(t) =& \frac{T_{l+2}''(t)}{4(l+1)(l+2)} - \frac{T_l''(t)}{2(l+1)(l-1)} + \frac{T_{l-2}''(t)}{4(l-1)(l-2)} \text{ for } l\geq 3.
\end{split}
\end{equation}
The expressions for $v_1^{(i+1)}$ follows directly from \eqref{eq:structure_preserving_operator_D2} and using this result $v_0^{(i+1)}$ can be computed from \eqref{eq:structure_preserving_operator_D1}.
\section{Derivation of the extension step for $\sigma=\jmath \omega$}
\label{appendix:extension_step_jomega}
Let $\phi_{i}$ and $\varphi_{i+1}$ be as defined in \eqref{eq:approach1_phi} and \eqref{eq:approach1_varphi}, then the equality $\mappedoperator\varphi_{i+1} = \phi_{i}$ becomes
\begin{multline*}
\left[\frac{2\jmath\omega}{\tau_K} \sum_{l=1}^{N_i+1} v_l^{(i+1)} T_l'\left(\frac{\theta}{\tau_K}\right) + \sum_{l=2}^{N_i+1} \frac{v_l^{(i+1)}}{\tau_K^2}T_l''\left(\frac{\theta}{\tau_K}\right)\right] e^{\jmath \omega\theta} + \\ \left[-\frac{2\jmath\omega}{\tau_K} \sum_{l=1}^{N_i+1} \overline{v_l^{(i+1)}} T_l'\left(\frac{\theta}{\tau_K}\right) + \sum_{l=2}^{N_i+1} \frac{\overline{v_l^{(i+1)}}}{\tau_K^2}T_l''\left(\frac{\theta}{\tau_K}\right)\right] e^{-\jmath \omega\theta} \\= \sum_{l=0}^{N_i} q_l^{(i)} T_l\left(\frac{\theta}{\tau_K}\right) e^{\jmath \omega \theta} + \sum_{l=0}^{N_i} \overline{q_l^{(i)}} T_l\left(\frac{\theta}{\tau_K}\right) e^{-\jmath \omega \theta}.
\end{multline*}
Matching the terms associated with $e^{\jmath \omega \theta}$ and $e^{-\jmath \omega \theta}$, gives the equality
\[
\frac{2\jmath\omega}{\tau_K} \sum_{l=1}^{N_i+1} v_l^{(i+1)} T_l'\left(\frac{\theta}{\tau_K}\right) + \sum_{l=2}^{N_i+1} \frac{v_l^{(i+1)}}{\tau_K^2}T_l''\left(\frac{\theta}{\tau_K}\right) = q_l^{(i)} T_l\left(\frac{\theta}{\tau_K}\right).
\]
Using \eqref{eq:chebyshev_derivative2} in combination with
\begin{equation*}
\begin{split}
T_1'(t) &= T_2''(t)/4 \text{ and }\\
T_l'(t) & = \frac{T_{l+1}''(t)}{2(l+1)}-\frac{T_{l-1}''(t)}{2(l-1)} \text{ for } l\geq 2
\end{split}
\end{equation*}
results in \eqref{eq:update_formula_approach1}. The expression for $v_0^{(i+1)}$ follows again from \eqref{eq:structure_preserving_operator_D1} and \eqref{eq:structure_preserving_operator_D2}.

\section*{Acknowledgments}
	The first author would like to thank P. Schwerdtner, M. Voigt and V.Mehrmann for the fruitful discussions. \newline
This work was supported by the project C14/17/072 of the KU Leuven
Research Council and the project G092721N of the Research Foundation-Flanders (FWO - Vlaanderen).

\bibliographystyle{siamplain}
\bibliography{ref}

\begin{thebibliography}{10}

\bibitem{Ahmed1968}
{\sc N.~Ahmed and P.~Fisher}, {\em {Study of algorithmic properties of
  chebyshev coefficients}}, International Journal of Computer Mathematics, 2
  (1968), pp.~307--317, \url{https://doi.org/10.1080/00207167008803043},
  \url{http://www.tandfonline.com/doi/abs/10.1080/00207167008803043}.

\bibitem{Benner}
{\sc P.~Benner, D.~Kressner, and V.~Mehrmann}, {\em {Skew-Hamiltonian and
  Hamiltonian Eigenvalue Problems: Theory, Algorithms and Applications}}, in
  Proceedings of the Conference on Applied Mathematics and Scientific
  Computing, Springer Netherlands, Dordrecht, the Netherlands, 2005, pp.~3--39,
  \url{https://doi.org/10.1007/1-4020-3197-1_1},
  \url{http://link.springer.com/10.1007/1-4020-3197-1_1}.

\bibitem{Boyd1990}
{\sc S.~Boyd and V.~Balakrishnan}, {\em {A regularity result for the singular
  values of a transfer matrix and a quadratically convergent algorithm for
  computing its L$\infty$-norm}}, Systems \& Control Letters, 15 (1990),
  pp.~1--7, \url{https://doi.org/10.1016/0167-6911(90)90037-U},
  \url{https://linkinghub.elsevier.com/retrieve/pii/016769119090037U}.

\bibitem{Bruinsma1990a}
{\sc N.~Bruinsma and M.~Steinbuch}, {\em {A fast algorithm to compute the
  H$\infty$-norm of a transfer function matrix}}, Systems \& Control Letters,
  14 (1990), pp.~287--293, \url{https://doi.org/10.1016/0167-6911(90)90049-Z},
  \url{https://linkinghub.elsevier.com/retrieve/pii/016769119090049Z}.

\bibitem{Gu2003a}
{\sc K.~Gu, V.~L. Kharitonov, and J.~Chen}, {\em {Stability of Time-Delay
  Systems}}, Birkh{\"{a}}user, Boston, MA, 2003,
  \url{https://doi.org/10.1007/978-1-4612-0039-0},
  \url{http://link.springer.com/10.1007/978-1-4612-0039-0}.

\bibitem{Gumussoy2011a}
{\sc S.~Gumussoy and W.~Michiels}, {\em {Fixed-Order H-Infinity Control for
  Interconnected Systems Using Delay Differential Algebraic Equations}}, SIAM
  Journal on Control and Optimization, 49 (2011), pp.~2212--2238,
  \url{https://doi.org/10.1137/100816444},
  \url{http://epubs.siam.org/doi/10.1137/100816444}.

\bibitem{guttel2017}
{\sc S.~G{\"{u}}ttel and F.~Tisseur}, {\em {The nonlinear eigenvalue problem}},
  Acta Numerica, 26 (2017), pp.~1--94.

\bibitem{nleigs}
{\sc S.~G\"{u}ttel, R.~Van~Beeumen, K.~Meerbergen, and W.~Michiels}, {\em
  {NLEIGS}: A class of fully rational {K}rylov methods for nonlinear eigenvalue
  problems}, SIAM Journal on Scientific Computing, 36 (2014), pp.~A2842--A2864,
  \url{https://doi.org/10.1137/130935045}.

\bibitem{Hale1977}
{\sc J.~Hale}, {\em {Theory of functional differential equations}}, vol.~3 of
  Applied Mathematical Sciences, Springer, New York, NY, 1977.

\bibitem{Hale1993a}
{\sc J.~K. Hale and S.~M.~V. Lunel}, {\em {Introduction to Functional
  Differential Equations}}, vol.~99 of Applied Mathematical Sciences, Springer,
  New York, NY, 1993, \url{https://doi.org/10.1007/978-1-4612-4342-7},
  \url{http://link.springer.com/10.1007/978-1-4612-4342-7}.

\bibitem{Hinrichsen2005}
{\sc D.~Hinrichsen and A.~J. Pritchard}, {\em {Mathematical Systems Theory I}},
  vol.~48 of Texts in Applied Mathematics, Springer, Berlin Heidelberg,
  Germany, 2005, \url{https://doi.org/10.1007/b137541},
  \url{http://link.springer.com/10.1007/b137541}.

\bibitem{Jarlebring2018}
{\sc E.~Jarlebring, M.~Bennedich, G.~Mele, E.~Ringh, and P.~Upadhyaya}, {\em
  {NEP-PACK: A Julia package for nonlinear eigenproblems - v0.2}},
  arXiv[1811.09592],  (2018), \url{http://arxiv.org/abs/1811.09592},
  \url{https://arxiv.org/abs/1811.09592}.

\bibitem{Jarlebring2010}
{\sc E.~Jarlebring, K.~Meerbergen, and W.~Michiels}, {\em {A Krylov Method for
  the Delay Eigenvalue Problem}}, SIAM Journal on Scientific Computing, 32
  (2010), pp.~3278--3300, \url{https://doi.org/10.1137/10078270X},
  \url{http://epubs.siam.org/doi/10.1137/10078270X}.

\bibitem{Jarlebring2012}
{\sc E.~Jarlebring, W.~Michiels, and K.~Meerbergen}, {\em {A linear eigenvalue
  algorithm for the nonlinear eigenvalue problem}}, Numerische Mathematik, 122
  (2012), pp.~169--195, \url{https://doi.org/10.1007/s00211-012-0453-0}.

\bibitem{Mehrmann2012}
{\sc V.~Mehrmann, C.~Schr{\"{o}}der, and V.~Simoncini}, {\em {An
  implicitly-restarted Krylov subspace method for real symmetric/skew-symmetric
  eigenproblems}}, Linear Algebra and Its Applications, 436 (2012),
  pp.~4070--4087, \url{https://doi.org/10.1016/j.laa.2009.11.009},
  \url{http://dx.doi.org/10.1016/j.laa.2009.11.009}.

\bibitem{Mehrmann2001}
{\sc V.~Mehrmann and D.~Watkins}, {\em {Structure-Preserving Methods for
  Computing Eigenpairs of Large Sparse Skew-Hamiltonian/Hamiltonian Pencils}},
  SIAM Journal on Scientific Computing, 22 (2001), pp.~1905--1925,
  \url{https://doi.org/10.1137/S1064827500366434},
  \url{http://epubs.siam.org/doi/10.1137/S1064827500366434}.

\bibitem{Mehrmann2002}
{\sc V.~Mehrmann and D.~Watkins}, {\em {Polynomial eigenvalue problems with
  Hamiltonian structure}}, Electronic Transactions on Numerical Analysis, 13
  (2002), pp.~106--118.

\bibitem{Mehrmann1991}
{\sc V.~L. Mehrmann}, {\em {The autonomous linear quadratic control problem :
  theory and numerical solution}}, Springer, Berlin Heidelberg, Germany, 1991.

\bibitem{Michiels2010}
{\sc W.~Michiels and S.~Gumussoy}, {\em {Characterization and Computation of
  $\mathcal{H}_{\infty}$ Norms for Time-Delay Systems}}, SIAM Journal on Matrix
  Analysis and Applications, 31 (2010), pp.~2093--2115,
  \url{https://doi.org/10.1137/090758751},
  \url{http://epubs.siam.org/doi/10.1137/090758751}.

\bibitem{Michiels2008}
{\sc W.~Michiels, E.~Jarlebring, and K.~Meerbergen}, {\em {Krylov-Based Model
  Order Reduction of Time-delay Systems}}, SIAM Journal on Matrix Analysis and
  Applications, 32 (2011), pp.~1399--1421,
  \url{https://doi.org/10.1137/100797436},
  \url{http://epubs.siam.org/doi/10.1137/100797436}.

\bibitem{Michiels2014}
{\sc W.~Michiels and S.-I. Niculescu}, {\em {Stability, Control, and
  Computation for Time-Delay Systems}}, Society for Industrial and Applied
  Mathematics, Philadelphia, PA, 2014,
  \url{https://doi.org/10.1137/1.9781611973631},
  \url{http://epubs.siam.org/doi/book/10.1137/1.9781611973631}.

\bibitem{Saad2003}
{\sc Y.~Saad}, {\em {Iterative Methods for Sparse Linear Systems}}, Society for
  Industrial and Applied Mathematics, Philadelphia, PA, 2003,
  \url{https://doi.org/10.1137/1.9780898718003},
  \url{http://epubs.siam.org/doi/book/10.1137/1.9780898718003}.

\bibitem{Trefethen2013}
{\sc L.~N. Trefethen}, {\em {Approximation Theory and Approximation Practice}},
  Society for Industrial and Applied Mathematics, Philadelphia, PA, 2013.

\bibitem{VanLoan1984}
{\sc C.~{Van Loan}}, {\em {A symplectic method for approximating all the
  eigenvalues of a Hamiltonian matrix}}, Linear Algebra and its Applications,
  61 (1984), pp.~233--251, \url{https://doi.org/10.1016/0024-3795(84)90034-X},
  \url{https://linkinghub.elsevier.com/retrieve/pii/002437958490034X}.

\bibitem{Voss2013}
{\sc H.~Voss}, {\em {Nonlinear Eigenvalue Problems}}, in Handbook of Linear
  Algebra, L.~Hogben, ed., Chapman and Hall/CRC, Boca Raton, FL, 2nd~ed., 2014,
  pp.~115:1--115:24, \url{https://doi.org/10.1201/b16113-70}.

\bibitem{Zhou1998}
{\sc K.~Zhou and J.~C. Doyle}, {\em {Essentials of robust control}}, Prentice
  Hall, Upper Saddle River, NJ, 1998.

\end{thebibliography}
\end{document}